\documentclass[11pt,reqno]{amsart}

\usepackage{amsmath,amssymb,amsthm}
\usepackage{graphicx}
\usepackage{mathrsfs}
\usepackage[margin=1in]{geometry}
\usepackage{hyperref}
\hypersetup{colorlinks=true,linkcolor=blue,citecolor=blue,urlcolor=blue}

\theoremstyle{plain}
\newtheorem{theorem}{Theorem}
\newtheorem{proposition}{Proposition}
\newtheorem{lemma}{Lemma}
\newtheorem{corollary}{Corollary}
\theoremstyle{definition}
\newtheorem{definition}{Definition}
\newtheorem{example}{Example}
\theoremstyle{remark}
\newtheorem{remark}{Remark}
\newtheorem{note}{Note}

\newcommand{\R}{\mathbb{R}}
\newcommand{\Z}{\mathbb{Z}}
\newcommand{\C}{\mathbb{C}}
\newcommand{\N}{\mathbb{N}}
\newcommand{\HH}{\mathbb{H}}
\newcommand{\Hil}{\mathscr{H}}
\newcommand{\Fou}{\mathscr{F}}
\newcommand{\D}{\mathscr{D}}
\newcommand{\ip}[2]{\langle #1,#2\rangle}
\newif\ifhavebbm
\IfFileExists{bbm.sty}{\havebbmtrue}{\havebbmfalse}
\ifhavebbm\usepackage{bbm}\fi
\ifhavebbm
  \newcommand{\indf}[1]{\mathbbm{1}_{#1}}
\else
  \newcommand{\indf}[1]{\chi_{#1}}
\fi
\DeclareMathOperator{\Lis}{Li}
\renewcommand{\Re}{\operatorname{Re}}

\begin{document}

\title[Dyadic resolvent representations and spectral data]
{Dyadic resolvent representations of self-adjoint operators:\\
propagator expansions, spectral measures, and zeta functions}

\author{N. Castillo}
\address{Independent researcher, Columbus, OH, USA}
\email{nickandpreeti@gmail.com}

\date{\today}

\begin{abstract}
For a self-adjoint operator $A$ on a Hilbert space, the dyadic resolvent
representation of Castillo, Costin and Costin~\cite{CCC} expresses the resolvent
$R_A(i\lambda)=(A-i\lambda)^{-1}$ as a series in the unitary group
$U_t=e^{-itA}$ sampled at the dyadic times $\{2^{-k}\}_{k\ge0}$. We develop this
representation structurally and through several spectral applications. We first
exhibit the underlying operators as a one-parameter scale family obeying a Landen
doubling recursion whose telescoping recovers the representation, and record its
readings as a series of Zak transforms and as a dyadic filter bank. As a worked
instance of the Zak-transform reading we obtain dyadic--Bessel series for the
lattice Green functions of $\Z^d$, with closed-form special values---the
lemniscatic constant $\Gamma(\tfrac14)$ in two dimensions and Watson's integral in
three. For the
Laplacian $A=-\Delta$ on $\R^n$ we expand $R_A(i\lambda)\varphi$, for
$\varphi\in L^1(\R^n)\cap L^2(\R^n)$, as a series of convolutions with the free
Schr\"odinger propagator, and derive explicit dyadic representations of the
fundamental solutions of the Laplace and Poisson equations in $\R^3$ and of the
one-dimensional heat equation. Finally, we reconstruct spectral data of $A$ from
the dyadic representation: the spectral measures on $\sigma(A)$---including,
through the limiting absorption principle, the absolutely continuous spectral
density of $-\Delta+V$---together with the density of states and the spectral
zeta function, the last reducing to the Riemann zeta function for $-\Delta$ on
the circle and yielding the functional determinant of $-\Delta+m^2$ there. We close by showing that the dyadic samples determine $A$
uniquely---an exact anti-aliasing of the propagator---so that all of this spectral
data is a function of the dyadic samples alone.
\end{abstract}

\maketitle

\section{Introduction}\label{sec:intro}

We use the dyadic resolvent representation of Castillo, Costin and
Costin~\cite{CCC}---which expresses the resolvent of a (possibly unbounded)
self-adjoint operator as a series in the unitary group $U_t=e^{-itA}$ sampled at
the dyadic times $\{2^{-k}\}_{k\ge0}$---to study resolvents, propagators, and
spectral data of self-adjoint operators. Parts of the development below were carried out in the author's doctoral dissertation~\cite{CastilloThesis}. In \S\ref{sec:zak} we read the resolvent
matrix coefficients $\ip{R_A(i\lambda)\varphi}{\psi}$ as a series of Zak
transforms (Definition~\ref{def:zak}) of the matrix coefficients of the unitary
group at dyadic sampling times (Proposition~\ref{prop:zakcoeff}); as a worked
instance we derive dyadic--Bessel series for the lattice Green functions of $\Z^d$
and their closed-form special values, including the lemniscatic constant
$\Gamma(\tfrac14)$ in two dimensions and Watson's integral in three
(Proposition~\ref{prop:lattice}, Corollary~\ref{cor:watson}). Letting these
sampling parameters vary, we isolate in \S\ref{sec:scale} a one-parameter family
of propagator averages obeying a Landen-type doubling recursion whose telescoping
returns the dyadic identity \eqref{eq:dyadic}
(Theorem~\ref{thm:telescope}, an independent derivation of
Proposition~\ref{prop:dyadic}) and which realizes the resolvent as
a dyadic filter bank (Remark~\ref{rem:filterbank}). In \S\ref{sec:laplacian} we compute the resolvent of the
Laplacian on $\R^n$ and show (Proposition~\ref{prop:laplacian}) that its action on
$\varphi\in L^1(\R^n)\cap L^2(\R^n)$ is a series of convolutions with the free
propagator $P_0(x;t)$, and in \S\ref{sec:fundamental} we apply the scalar form of
the dyadic identity to obtain explicit dyadic representations of the fundamental
solutions of the Laplace and Poisson equations in $\R^3$ and of the one-dimensional heat
equation. Finally, \S\ref{sec:spectral} reconstructs spectral data of $A$ from the
dyadic representation: the spectral measures on $\sigma(A)$---including, through
the limiting absorption principle, the absolutely continuous spectral density of
$-\Delta+V$---together with the density of states and the spectral zeta function,
the latter split by \eqref{eq:zeta} into an entire part built from integer-time
heat samples and a dyadic correction carrying all the poles
(Proposition~\ref{prop:zeta}), and
recovering the Riemann zeta function for $-\Delta$ on the circle and
the functional determinant of $-\Delta+m^{2}$ there
(Example~\ref{ex:zetadet}). We
then record that the dyadic samples determine $A$ uniquely---an exact
anti-aliasing of the propagator (Corollary~\ref{cor:identifiability}).

Throughout, the dyadic identity is used as a machine that converts propagator
samples into spectral data. The closed forms reached along the way---the lattice
Green functions and their $\Gamma(\tfrac14)$ and Watson values, the Faddeeva
function, the functional determinant on the circle---are classical, and each is
reached faster by classical means; they serve here as a calibration set against
which the reconstruction is tested, not as new evaluations. It
has been pointed out by O.~Costin that some of these results may prove useful in
understanding Feynman path integrals.

We first recall the result of \cite{CCC} on which everything below rests.

\begin{proposition}[{\cite[Proposition~22(i)]{CCC}}]\label{prop:dyadic}
Let $\Hil$ be a Hilbert space and $A$ a bounded or unbounded self-adjoint
operator. Let $U_t$ be the unitary evolution operator generated by $A$,
$U_t=e^{-itA}$. If $\lambda\in\R^+$, then
\begin{equation}\label{eq:dyadic}
\begin{aligned}
(A-i\lambda)^{-1}
&= i\,(1-e^{-\lambda}U_1)^{-1}
   - i\sum_{k=1}^{\infty}\frac{1}{2^{k}}\bigl(1+e^{-\lambda/2^{k}}U_{2^{-k}}\bigr)^{-1}\\
&= i\sum_{j=0}^{\infty}e^{-j\lambda}U_{j}
   - i\lim_{\ell\to\infty}\sum_{k=1}^{\ell}\sum_{j=0}^{\infty}
     (-1)^{j}\,2^{-k}\,e^{-j\lambda/2^{k}}\,U_{j2^{-k}}.
\end{aligned}
\end{equation}
Convergence holds in the strong operator topology. For $\lambda<0$ one simply
complex-conjugates~\eqref{eq:dyadic}. (The limits cannot, in general, be
interchanged.)
\end{proposition}

The two tools behind \eqref{eq:dyadic} are the dyadic expansion of the Cauchy
kernel~\cite{CCC} and the spectral theorem for self-adjoint operators
\cite{SimonOT}; we refer
the reader to \cite{CCC} for the proof. The second equality follows from the
first by expanding each resolvent in its (operator-norm convergent) geometric
series, using $\|e^{-\lambda}U_1\|=e^{-\lambda}<1$, $\|e^{-\lambda/2^{k}}U_{2^{-k}}\|
=e^{-\lambda/2^{k}}<1$ and the group law $U_t^{\,j}=U_{jt}$.

\begin{remark}\label{rem:average}
For fixed $\lambda>0$, \eqref{eq:dyadic} constructs $R_A(i\lambda)$ as a weighted
average of resolvents of samples $U_{2^{-k}}$ of the unitary propagator
(see~\eqref{eq:weighted}). These samples occur at shorter and shorter times
$2^{-k}$ following $t=0$: in a precise sense the behavior of $R_A(i\lambda)$ is
controlled by the behavior of the system shortly after the clock is started,
all information being carried by the dyadic time set $\{2^{-k}\}_{k\ge0}$.
\end{remark}

\section{Matrix coefficients of \texorpdfstring{$R_A(i\lambda)$}{R\_A(i lambda)} and the Zak transform}
\label{sec:zak}

Pairing the geometric-series form of \eqref{eq:dyadic} against vectors
$\varphi,\psi\in\Hil$ and using the continuity of the inner product together
with strong-operator convergence gives
\begin{equation}\label{eq:paired}
\begin{aligned}
\ip{R_A(i\lambda)\varphi}{\psi}
&= i\sum_{j=0}^{\infty}e^{-j\lambda}\ip{U_j\varphi}{\psi}\\
&\quad{}-i\lim_{\ell\to\infty}\sum_{k=1}^{\ell}\sum_{j=0}^{\infty}
   (-1)^{j}\,2^{-k}\,e^{-j\lambda/2^{k}}\,\ip{U_{j2^{-k}}\varphi}{\psi}.
\end{aligned}
\end{equation}
Throughout, $\ip{\cdot}{\cdot}$ is linear in the first argument and antilinear in
the second. We now recognize the terms on the right of \eqref{eq:paired} as Zak
transforms (at varying sampling times) of the matrix coefficients
$t\mapsto\ip{U_t\varphi}{\psi}$.

\begin{definition}[Zak transform; cf.\ \cite{Grochenig}]\label{def:zak}
Let $f\in C_c(\R)$, $T>0$, and let $t,w$ be real variables. The function
\begin{equation}\label{eq:zakdef}
(Z_T f)(t,w):=\sqrt{T}\sum_{n\in\Z}f(t+nT)\,e^{-2\pi i n w},
\qquad t,w\in\R,
\end{equation}
is called the \emph{Zak transform} of $f$ with sampling time $T$.
\end{definition}

As the second variable $w$ varies, the expansion basis
$\{e^{-2\pi i n w}\}_{n\in\Z}$ changes; in this sense $Z_T$ produces a whole
family of expansions. Formula \eqref{eq:zakdef} is defined pointwise for
$f\in C_c(\R)$ and extends to $L^2(\R)$:

\begin{lemma}[\cite{Grochenig}]\label{lem:zakunitary}
For $T=1$ the Zak transform $Z_1$ of \eqref{eq:zakdef} is a unitary map of
$L^2(\R)$ onto $L^2([0,1]^2)$. For general $T>0$, $Z_T$ satisfies the
quasi-periodicity relations
\[
(Z_Tf)(t+T,w)=e^{2\pi i w}(Z_Tf)(t,w),\qquad (Z_Tf)(t,w+1)=(Z_Tf)(t,w),
\]
and $\|Z_Tf\|_{L^2([0,T]\times[0,1])}=\sqrt{T}\,\|f\|_{L^2(\R)}$; equivalently
$T^{-1/2}Z_T$ is unitary from $L^2(\R)$ onto $L^2([0,T]\times[0,1])$.
\end{lemma}

There are many properties of $Z_T$ reminiscent of the discrete Fourier
transform, with applications in signal processing, quantum field theory and
algebraic geometry; cf.\ \cite{Janssen}, \cite{FollandHAPS}. Like the Fourier
transform $\Fou$, the Zak transform acts on Gaussians in a special way; however,
rather than Gaussians being eigenfunctions, $Z_T$ multiplies the Gaussian by a
Jacobi theta function.

\begin{example}[Zak transform of a Gaussian; cf.\ \cite{Janssen}]\label{ex:gaussian}
Let $\gamma>0$ and $f(t)=(2\gamma)^{1/4}e^{-\pi\gamma t^{2}}$. A routine
calculation gives
\begin{equation}\label{eq:gaussian}
(Z_1 f)(t,w)=(2\gamma)^{1/4}e^{-\pi\gamma t^{2}}\,
\theta_3\!\bigl(w-i\gamma t,\,e^{-\pi\gamma}\bigr)
= f(t)\,\theta_3\!\bigl(w-i\gamma t,\,e^{-\pi\gamma}\bigr),
\end{equation}
where $\theta_3$ is the third Jacobi theta function,
\[
\theta_3(z;q)=\sum_{k=-\infty}^{\infty}q^{k^{2}}e^{-2\pi i k z},
\qquad z\in\C,\ q=e^{\pi i\tau},\ \tau\in\HH .
\]
\end{example}

Several algebraic relations for $Z_T$ become apparent when its inputs are varied
separately, making $Z_T$ a natural object in the study of the Heisenberg group
\cite{AlAmeerKisil} and Gabor systems \cite{Grochenig,FeichtingerGrochenig}.
Returning to
\eqref{eq:paired} we have the following.

\begin{proposition}\label{prop:zakcoeff}
Let $\Hil$ be a Hilbert space and $A$ a bounded or unbounded self-adjoint
operator, $U_t=e^{-itA}$. Given $\varphi,\psi\in\Hil$ and $\lambda\in\R^{+}$,
define the $\C$-valued function $G^{\lambda}_{\varphi,\psi}$ on $\R$ by
\begin{equation}\label{eq:Gdef}
G^{\lambda}_{\varphi,\psi}(t)=i\,\ip{U_t\varphi}{\psi}\,e^{-\lambda t}\,e^{2\pi i t}\,
\indf{[0,\infty)}(t).
\end{equation}
The superscript records the dependence on $\lambda$ through the factor
$e^{-\lambda t}$. Then the matrix coefficient of $R_A(i\lambda)$ with respect to $\varphi,\psi$
satisfies
\begin{equation}\label{eq:zakcoeff}
\ip{R_A(i\lambda)\varphi}{\psi}
=(Z_1 G^{\lambda}_{\varphi,\psi})(0,1)
-\lim_{\ell\to\infty}\sum_{k=1}^{\ell}2^{-k/2}
   (Z_{2^{-k}}G^{\lambda}_{\varphi,\psi})\!\left(0,\,2^{-k}(1+2^{k-1})\right).
\end{equation}
The series converges in the sense of Proposition~\ref{prop:dyadic}; when $A$ is
bounded it converges absolutely.
\end{proposition}

\begin{remark}\label{rem:variants}
One could just as well use $K^{\lambda}_{\varphi,\psi}(t)=i\ip{U_{|t|}\varphi}{\psi}
e^{-\lambda|t|}e^{2\pi it}$ or $F^{\lambda}_{\varphi,\psi}(t)=i\ip{U_t\varphi}{\psi}
e^{-\lambda|t|}e^{2\pi it}$ in place of $G^{\lambda}_{\varphi,\psi}$; each is continuous
from the outset. With $F^{\lambda}_{\varphi,\psi}$ there would be a part of the expansion
in which $\psi$ is acted on by $U_t$, so the resulting formula would differ
slightly from \eqref{eq:zakcoeff}.
\end{remark}

\begin{proof}
By the Cauchy--Schwarz inequality and unitarity of $U_t$,
\begin{equation}\label{eq:Gbound}
|G^{\lambda}_{\varphi,\psi}(t)|
=\bigl|i\ip{U_t\varphi}{\psi}e^{-\lambda t}e^{2\pi it}\bigr|\indf{[0,\infty)}(t)
\le \|U_t\|_{\mathrm{op}}\|\varphi\|\,\|\psi\|\,e^{-\lambda t}\indf{[0,\infty)}(t)
= C_{\varphi,\psi}\,e^{-\lambda t}\indf{[0,\infty)}(t),
\end{equation}
with $C_{\varphi,\psi}:=\|\varphi\|\,\|\psi\|$ since $\|U_t\|_{\mathrm{op}}=1$.
Hence $G^{\lambda}_{\varphi,\psi}\in L^p(\R)$ for every $p\in[1,\infty]$, and
\[
\|G^{\lambda}_{\varphi,\psi}\|_{L^2(\R)}
\le C_{\varphi,\psi}\Bigl(\int_0^\infty e^{-2\lambda t}\,dt\Bigr)^{1/2}
=\frac{C_{\varphi,\psi}}{\sqrt{2\lambda}} .
\]
In particular, by Lemma~\ref{lem:zakunitary}, the Zak transforms $Z_1G^{\lambda}_{\varphi,\psi}$
and $Z_{2^{-k}}G^{\lambda}_{\varphi,\psi}$ are well-defined $L^2$ functions on their
respective fundamental domains.

Because $G^{\lambda}_{\varphi,\psi}$ decays exponentially, for every $T>0$ and every
$(t,w)$ the defining series of $(Z_T G^{\lambda}_{\varphi,\psi})(t,w)$ converges
absolutely; the resulting values are therefore genuine pointwise values (not
merely defined a.e.) and may be computed term by term. Using
$G^{\lambda}_{\varphi,\psi}(\tau)=i\ip{U_\tau\varphi}{\psi}e^{-\lambda\tau}e^{2\pi i\tau}$
for $\tau\ge0$ and $G^{\lambda}_{\varphi,\psi}(\tau)=0$ for $\tau<0$, and noting
$e^{2\pi i n}=1$,
\begin{equation}\label{eq:firstterm}
(Z_1 G^{\lambda}_{\varphi,\psi})(0,1)
=\sum_{n\in\Z}G^{\lambda}_{\varphi,\psi}(n)
=i\sum_{j=0}^{\infty}e^{-j\lambda}\ip{U_j\varphi}{\psi},
\end{equation}
which is precisely the first sum in \eqref{eq:paired}. Similarly, with
$T=2^{-k}$ and $w_k=2^{-k}(1+2^{k-1})=2^{-k}+\tfrac12$, the phase factor is
$e^{2\pi i n(2^{-k}-w_k)}=e^{-\pi i n}=(-1)^n$, whence
\begin{equation}\label{eq:kterm}
2^{-k/2}(Z_{2^{-k}}G^{\lambda}_{\varphi,\psi})(0,w_k)
=i\,2^{-k}\sum_{j=0}^{\infty}(-1)^{j}e^{-j\lambda/2^{k}}\ip{U_{j2^{-k}}\varphi}{\psi}.
\end{equation}
Comparing \eqref{eq:kterm} with the dyadic part of \eqref{eq:paired}, the right
side of \eqref{eq:kterm} equals the \emph{negative} of the $k$-th summand there.
Combining \eqref{eq:firstterm}, \eqref{eq:kterm} and \eqref{eq:paired} yields
\eqref{eq:zakcoeff}, with convergence inherited from
Proposition~\ref{prop:dyadic}.

Finally, if $A$ is bounded, say $\sigma(A)\subset[-M,M]$, then by the spectral
theorem the $k$-th summand of \eqref{eq:zakcoeff} has modulus
\[
2^{-k}\sup_{|x|\le M}\frac{1}{\bigl|1+e^{-(\lambda+ix)2^{-k}}\bigr|}\,
\|\varphi\|\,\|\psi\|
\;\sim\;2^{-k}\,\frac{\|\varphi\|\,\|\psi\|}{2}\qquad(k\to\infty),
\]
since $e^{-(\lambda+ix)2^{-k}}\to1$ uniformly on $|x|\le M$; the series therefore
converges absolutely. (For unbounded $A$ the denominator above is not bounded
below uniformly in $x$, and only strong-operator convergence is asserted; see
Remark~\ref{rem:absconv}.)
\end{proof}

\begin{remark}\label{rem:absconv}
A direct norm estimate does not yield absolute convergence of \eqref{eq:zakcoeff}
for unbounded $A$, for two reasons. First, the Zak isometry holds on the
fundamental domain $[0,T]\times[0,1]$, not on $[0,1]^2$, and carries the factor
$\sqrt{T}$ (Lemma~\ref{lem:zakunitary}); with $T=2^{-k}$ this rescales the
would-be constant. Second, and more essentially, a triangle-inequality bound
discards the alternating cancellation in the inner sum of \eqref{eq:kterm}, which
is exactly what makes the $k$-series summable. Convergence is therefore
established here by inheriting it from Proposition~\ref{prop:dyadic} rather than
by a direct norm estimate, and absolute convergence is claimed only in the
bounded case.
\end{remark}

The mechanism of Proposition~\ref{prop:zakcoeff} becomes fully explicit when the
spectral data are Gaussian, in which case Example~\ref{ex:gaussian} applies almost
verbatim and the resolvent coefficients close up in theta functions.

\begin{example}[A Gaussian spectral density: resolvent coefficients as partial
theta functions]\label{ex:gausspartialtheta}
Take $\Hil=L^2(\R)$ and let $A$ be multiplication by $x$, a self-adjoint operator
with $\sigma(A)=\R$; then $U_t=e^{-itA}$ is multiplication by $e^{-itx}$ and
$R_A(i\lambda)$ is multiplication by $(x-i\lambda)^{-1}$. Fix $\gamma>0$ and choose
$\varphi=\psi=e^{-\pi\gamma x^{2}/2}$, so that the spectral density
$\varphi\bar\psi=e^{-\pi\gamma x^{2}}$ is Gaussian. The group coefficient is then
itself Gaussian in $t$,
\begin{equation}\label{eq:gausscoeff}
\ip{U_t\varphi}{\psi}=\int_\R e^{-itx}e^{-\pi\gamma x^{2}}\,dx
=\gamma^{-1/2}e^{-t^{2}/(4\pi\gamma)} .
\end{equation}
This is the general mechanism at work: $\ip{U_t\varphi}{\psi}$ is the Fourier
transform of the spectral measure $\mu_{\varphi,\psi}$, hence Gaussian exactly
when $\mu_{\varphi,\psi}$ is.

Substituting \eqref{eq:gausscoeff} into \eqref{eq:Gdef} and completing the square,
$-t^{2}/(4\pi\gamma)-\lambda t=-(t-t_0)^{2}/(4\pi\gamma)+\pi\gamma\lambda^{2}$ with
$t_0=-2\pi\gamma\lambda$, exhibits $G^{\lambda}_{\varphi,\psi}$ as a shifted, modulated
Gaussian restricted to the half-line,
\begin{equation}\label{eq:Galmostgauss}
G^{\lambda}_{\varphi,\psi}(t)
=i\gamma^{-1/2}e^{\pi\gamma\lambda^{2}}\,
e^{-(t-t_0)^{2}/(4\pi\gamma)}\,e^{2\pi it}\,\indf{[0,\infty)}(t).
\end{equation}
But for the cutoff $\indf{[0,\infty)}$ this is exactly the input of
Example~\ref{ex:gaussian}, whose Zak transform is a Jacobi $\theta_3$; the
translation $t\mapsto t-t_0$ and the modulation $e^{2\pi it}$ only relabel its
arguments. The cutoff restricts the lattice sums in \eqref{eq:zakcoeff}---all
evaluated at $t=0$---to $n\ge0$, replacing the bilateral theta by the one-sided
\emph{partial theta function} $\theta_{+}(q,z):=\sum_{n\ge0}q^{n^{2}}z^{n}$.
Indeed, with $q_T=e^{-T^{2}/(4\pi\gamma)}$,
\[
(Z_T G^{\lambda}_{\varphi,\psi})(0,w)
=i\sqrt{T/\gamma}\;\theta_{+}\!\bigl(q_T,\,e^{-\lambda T}e^{2\pi i(T-w)}\bigr),
\]
and at the sampling points of \eqref{eq:zakcoeff} the second argument is
$e^{-\lambda}$ for the leading term and $-e^{-\lambda2^{-k}}$ for the $k$-th, so
\begin{equation}\label{eq:gausspartialtheta}
\ip{R_A(i\lambda)\varphi}{\psi}
=i\gamma^{-1/2}\Bigl[\theta_{+}\!\bigl(q_1,e^{-\lambda}\bigr)
-\lim_{\ell\to\infty}\sum_{k=1}^{\ell}2^{-k}\,
\theta_{+}\!\bigl(q_{2^{-k}},-e^{-\lambda2^{-k}}\bigr)\Bigr].
\end{equation}
The left side is the Cauchy transform of a Gaussian, a Faddeeva value:
$\ip{R_A(i\lambda)\varphi}{\psi}
=i\pi\,e^{\pi\gamma\lambda^{2}}\operatorname{erfc}\!\bigl(\lambda\sqrt{\pi\gamma}\bigr)$.
Thus \eqref{eq:gausspartialtheta} is a dyadic partial-theta representation of the
complementary error function; for $\gamma=\lambda=1$ both sides equal
$i\pi e^{\pi}\operatorname{erfc}(\sqrt\pi)=0.8861\ldots\,i$, which we have checked
numerically. Recovering the full Jacobi $\theta_3$ of Example~\ref{ex:gaussian}
would require a two-sided $G^{\lambda}_{\varphi,\psi}$; the causality of the dyadic sampling
$t=j2^{-k}\ge0$ is exactly what makes the representation one-sided.
\end{example}

\begin{remark}[Partial versus full theta]\label{rem:thetasplit}
Splitting the bilateral lattice sum at $n=0$ gives, for $0<|q|<1$ and $z\ne0$,
\[
\sum_{n\in\Z}q^{n^{2}}z^{n}=\theta_{+}(q,z)+\theta_{+}(q,z^{-1})-1,
\]
whose left side is the Jacobi theta of Example~\ref{ex:gaussian}, equal to
$\theta_3\!\bigl(\tfrac{i}{2\pi}\ln z;\,q\bigr)$ in the convention used there. Each
partial theta in \eqref{eq:gausspartialtheta} therefore splits as
\[
\theta_{+}(q,z)=\theta_3\!\Bigl(\tfrac{i}{2\pi}\ln z;\,q\Bigr)
-\sum_{m\ge1}q^{m^{2}}z^{-m},
\]
a genuine Jacobi theta---the Zak transform of the \emph{uncut} Gaussian
$\widetilde G^{\lambda}_{\varphi,\psi}$ obtained by dropping $\indf{[0,\infty)}$ from
\eqref{eq:Galmostgauss}---minus the complementary tail over the negative lattice
points $n\le-1$ that the cutoff discards. At the sampling points $|z|=e^{-\lambda T}<1$,
so the tail has $|z^{-1}|>1$ yet still converges through the Gaussian factor
$q^{m^{2}}$. Thus the coefficients of Example~\ref{ex:gausspartialtheta} may be
written with the full modular $\theta_3$ in place of $\theta_{+}$, at the cost of
an explicit, rapidly convergent correction.
\end{remark}

\begin{remark}[A shifted spectral density and the Faddeeva function on a line]
\label{rem:gaussshift}
Centering the density at $x_0\in\R$, i.e.\ $\varphi=\psi=e^{-\pi\gamma(x-x_0)^{2}/2}$,
multiplies the group coefficient \eqref{eq:gausscoeff} by the modulation
$e^{-itx_0}$, and hence multiplies the second argument of every partial theta in
\eqref{eq:gausspartialtheta} by $e^{-ix_0T}$:
\[
\ip{R_A(i\lambda)\varphi}{\psi}
=i\gamma^{-1/2}\Bigl[\theta_{+}\!\bigl(q_1,e^{-\lambda}e^{-ix_0}\bigr)
-\lim_{\ell\to\infty}\sum_{k=1}^{\ell}2^{-k}\,
\theta_{+}\!\bigl(q_{2^{-k}},-e^{-\lambda2^{-k}}e^{-ix_0 2^{-k}}\bigr)\Bigr].
\]
The left side is now the Cauchy transform of the shifted Gaussian, the full
Faddeeva value $i\pi\,w\bigl((i\lambda-x_0)\sqrt{\pi\gamma}\bigr)$ with
$w(\zeta)=e^{-\zeta^{2}}\operatorname{erfc}(-i\zeta)$. As $x_0$ ranges over $\R$
the argument $(i\lambda-x_0)\sqrt{\pi\gamma}$ sweeps the horizontal line
$\Im\zeta=\lambda\sqrt{\pi\gamma}$, so the family realizes $w$ along an entire line
through this one dyadic construction, the case $x_0=0$ of
Example~\ref{ex:gausspartialtheta} being where that line meets the imaginary axis.
We verified the identity numerically (e.g.\ $\gamma=\lambda=1$, $x_0=\tfrac{7}{10}$).
\end{remark}

\begin{proposition}[Lattice resolvent coefficients]\label{prop:lattice}
Let $H$ be the adjacency operator of $\Z^{d}$ on $\ell^{2}(\Z^{d};\C)$,
$(H\psi)(\mathbf n)=\sum_{|\mathbf m-\mathbf n|=1}\psi(\mathbf m)$, bounded and self-adjoint
with $\sigma(H)=[-2d,2d]$, and let $\{\delta_{\mathbf n}\}_{\mathbf n\in\Z^{d}}$ be the
standard orthonormal basis. Its propagator coefficients are
\begin{equation}\label{eq:phifamily}
\phi_{\mathbf n}(t):=\langle\delta_{\mathbf n},e^{-itH}\delta_{\mathbf 0}\rangle
=\prod_{i=1}^{d}(-i)^{n_i}J_{n_i}(2t),\qquad \mathbf n\in\Z^{d}.
\end{equation}
For every $\mathbf n$ and every $z=E+i\lambda$ with $E\in\R$ and $\lambda>0$, the resolvent coefficient
$G(\mathbf n;z):=\langle\delta_{\mathbf n},(z-H)^{-1}\delta_{\mathbf 0}\rangle$ is
\begin{equation}\label{eq:latticefamily}
G(\mathbf n;E+i\lambda)=-i\sum_{j\ge0}e^{-j\lambda}e^{ijE}\phi_{\mathbf n}(j)
+i\sum_{k\ge1}2^{-k}\sum_{j\ge0}(-1)^{j}e^{-j\lambda2^{-k}}e^{ij2^{-k}E}\phi_{\mathbf n}(j2^{-k}),
\end{equation}
the series converging absolutely, and the family $\{G(\mathbf n;z)\}_{\mathbf n\in\Z^{d}}$
satisfies the lattice resolvent equation
\begin{equation}\label{eq:latticerec}
z\,G(\mathbf n;z)-\sum_{|\mathbf m-\mathbf n|=1}G(\mathbf m;z)=\delta_{\mathbf n,\mathbf 0}.
\end{equation}
\end{proposition}

\begin{proof}
The lattice separates as $H=\sum_{i=1}^{d}H_{i}$, where $H_i$ denotes the
one-dimensional nearest-neighbor hopping operator acting in the $i$-th coordinate,
so $e^{-itH}=\prod_{i}e^{-itH_{i}}$ and
$\phi_{\mathbf n}(t)=\prod_{i}\langle\delta_{n_i},e^{-itH_{i}}\delta_{0}\rangle$. The
Jacobi--Anger expansion $e^{-2it\cos k}=\sum_{m\in\Z}(-i)^{m}J_{m}(2t)e^{imk}$ gives
$\langle\delta_{n},e^{-itH_{i}}\delta_{0}\rangle
=\frac1{2\pi}\int_{-\pi}^{\pi}e^{-ikn}e^{-2it\cos k}\,dk=(-i)^{n}J_{n}(2t)$, which is
\eqref{eq:phifamily}. By the same Fourier transform $\ell^2(\Z^d)\to L^2(\mathbb T^d)$,
$H$ is multiplication by the symbol $2\sum_{i=1}^{d}\cos k_i$, whose range as
$\mathbf k$ runs over the connected torus $\mathbb T^d$ is the full interval
$[-2d,2d]$; hence $\sigma(H)=[-2d,2d]$, purely absolutely continuous. Since $H$ is bounded, Proposition~\ref{prop:dyadic} applies to
$A=H-E$; the $\langle\delta_{\mathbf n},\,\cdot\,\delta_{\mathbf 0}\rangle$ matrix element
of \eqref{eq:dyadic}, with
$\langle\delta_{\mathbf n},e^{-it(H-E)}\delta_{\mathbf 0}\rangle=e^{itE}\phi_{\mathbf n}(t)$
and $G(\mathbf n;z)=-\langle\delta_{\mathbf n},(H-E-i\lambda)^{-1}\delta_{\mathbf 0}\rangle$,
gives \eqref{eq:latticefamily}; absolute convergence is the bounded-operator case of
Proposition~\ref{prop:dyadic}. Finally \eqref{eq:latticerec} is the
$\langle\delta_{\mathbf n},\,\cdot\,\delta_{\mathbf 0}\rangle$ element of $(z-H)(z-H)^{-1}=I$.
\end{proof}

\begin{remark}[The one-dimensional family]\label{rem:lattice1d}
For $d=1$ the coefficients are single Bessel functions, $\phi_{n}(t)=(-i)^{n}J_{n}(2t)$,
and \eqref{eq:latticefamily} is a dyadic--Bessel series for the elementary chain Green
function
\[
G(n;z)=\frac{1}{\sqrt{z^{2}-4}}\left(\frac{z-\sqrt{z^{2}-4}}{2}\right)^{\!|n|},
\]
the branch of $\sqrt{z^{2}-4}$ chosen so that $|z-\sqrt{z^{2}-4}|<2$. Each integer $n$ thus
furnishes a distinct closed-form identity, verified to the precision carried (twenty digits
at $|n|\le3$, $z=\tfrac35+\tfrac{7}{10}i$).
\end{remark}

\begin{example}[A lattice Green function and a $\Gamma(\tfrac14)$ value]
\label{ex:lattice}
The diagonal $d=2$ case of Proposition~\ref{prop:lattice} is the planar lattice Green
function $G(z)=G(\mathbf 0;z)=\langle\delta_{\mathbf 0},(z-H)^{-1}\delta_{\mathbf 0}\rangle$
on $\ell^{2}(\Z^{2})$, with $\phi_{\mathbf 0}(t)=J_{0}(2t)^{2}$, so
\eqref{eq:latticefamily} reads
\begin{equation}\label{eq:latticeG}
G(E+i\lambda)=-i\sum_{j\ge0}e^{-j\lambda}e^{ijE}J_{0}(2j)^{2}
+i\sum_{k\ge1}2^{-k}\sum_{j\ge0}(-1)^{j}e^{-j\lambda2^{-k}}e^{ij2^{-k}E}
J_{0}\!\bigl(j2^{1-k}\bigr)^{2}.
\end{equation}
This reproduces the classical closed form $G(z)=\tfrac{2}{\pi z}\,K(4/z)$, with $K$ the
complete elliptic integral of the first kind: across the resolvent set the two agree to
the thirteen digits tested. Figure~\ref{fig:lattice} overlays the reconstruction on the
exact $G$ along the line $E+\tfrac15 i$; the series tracks both the band-edge cusps of
$\Re G$ and the logarithmic van Hove peak of the density of states
$-\tfrac1\pi\Im G$ at $E=0$.

The representation also returns the lattice's special values. At $z=4\sqrt2$ the modulus
$4/z=1/\sqrt2$ is lemniscatic, $K(1/\sqrt2)=\Gamma(\tfrac14)^{2}/(4\sqrt\pi)$, so
\eqref{eq:latticeG} evaluates the return Green function to
\[
G(0,0;4\sqrt2)=\frac{\Gamma(\tfrac14)^{2}}{8\sqrt2\,\pi^{3/2}}=0.20865671041851\ldots,
\]
a $\Gamma(\tfrac14)$ period of the lattice. Where the multiplication operator of
Example~\ref{ex:gausspartialtheta} produced the Faddeeva function through a partial-theta
series, the lattice operator produces the elliptic Green function through a Bessel
series; both are diagonal resolvent coefficients assembled by the dyadic identity from a
closed-form propagator coefficient.
\end{example}

\begin{figure}[ht]
\centering
\includegraphics[width=\textwidth]{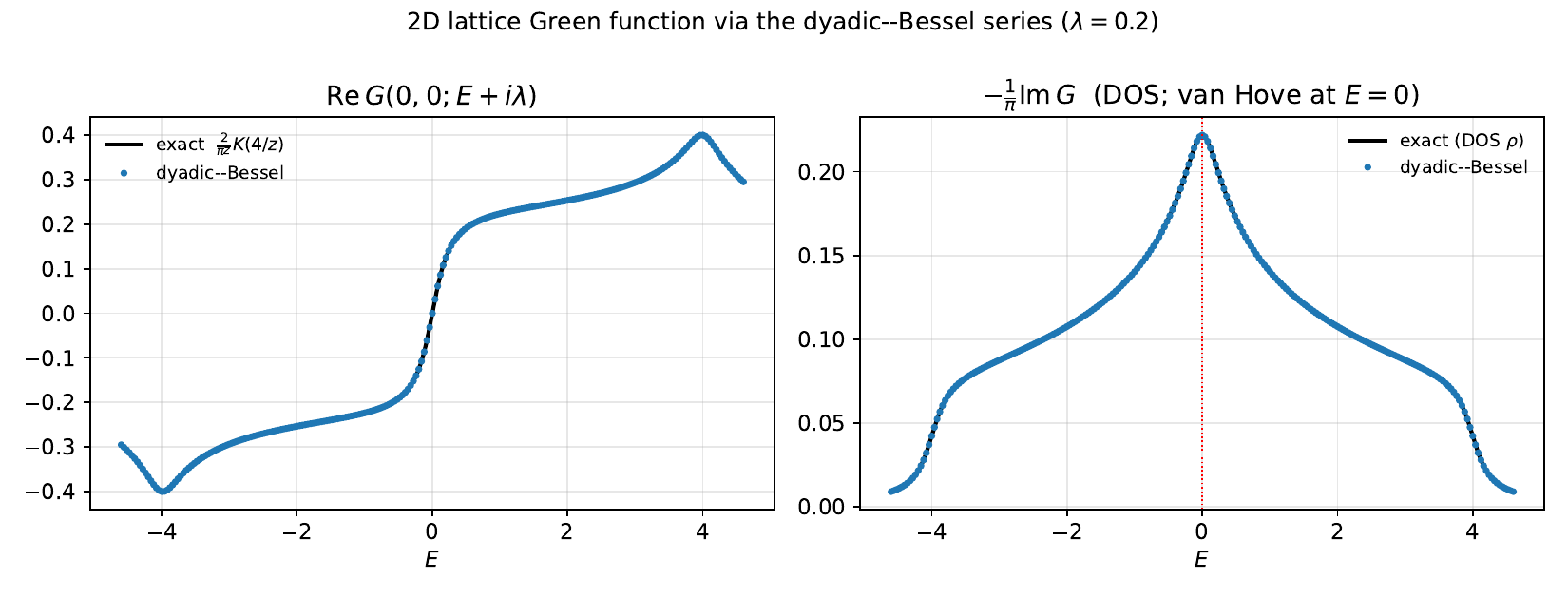}
\caption{The square-lattice Green function $G(0,0;E+i\lambda)$ at $\lambda=\tfrac15$
reconstructed from the dyadic--Bessel series \eqref{eq:latticeG} (markers) against the
closed form $\tfrac{2}{\pi z}K(4/z)$ (solid). Left: $\Re G$, with cusps at the band
edges $E=\pm4$. Right: the density of states $-\tfrac1\pi\Im G$, with the logarithmic
van Hove singularity at $E=0$.}
\label{fig:lattice}
\end{figure}

\begin{corollary}[Return Green functions and the Watson integral]\label{cor:watson}
The diagonal case $\mathbf n=\mathbf 0$ of Proposition~\ref{prop:lattice} represents the
return Green function of $\Z^{d}$ as the dyadic transform of $J_{0}(2t)^{d}$,
\begin{equation}\label{eq:return}
G(\mathbf 0;E+i\lambda)=-i\sum_{j\ge0}e^{-j\lambda}e^{ijE}J_{0}(2j)^{d}
+i\sum_{k\ge1}2^{-k}\sum_{j\ge0}(-1)^{j}e^{-j\lambda2^{-k}}e^{ij2^{-k}E}
J_{0}\!\bigl(j2^{1-k}\bigr)^{d},
\end{equation}
of which Example~\ref{ex:lattice} is the planar case $d=2$. In three dimensions
$J_{0}(2t)^{3}=O(t^{-3/2})$, so the leading sum in \eqref{eq:return} converges absolutely
up to the band edge $z=2d=6$, where the return Green function is half Watson's integral,
\[
G(\mathbf 0;6)=\frac{\sqrt6}{192\,\pi^{3}}\,
\Gamma\!\bigl(\tfrac1{24}\bigr)\Gamma\!\bigl(\tfrac5{24}\bigr)
\Gamma\!\bigl(\tfrac7{24}\bigr)\Gamma\!\bigl(\tfrac{11}{24}\bigr)=0.2527310\ldots
\]
Thus the lemniscatic value $\Gamma(\tfrac14)^{2}/(8\sqrt2\,\pi^{3/2})$ of
Example~\ref{ex:lattice} and the cubic Watson constant are two members of one family of
dyadic--Bessel representations, indexed by dimension through the power of $J_{0}(2t)$.
\end{corollary}

\section{A scale family of propagator averages and a Landen recursion}
\label{sec:scale}

In \S\ref{sec:zak} the resolvent coefficient was read off the Zak transform at
the fixed base point $t=0$. We now let the first argument vary. Since shifting
the first argument translates $G$, and $U_{\tau+t}=U_\tau U_t$, the translation
acts on the input vector, $\varphi\mapsto U_t\varphi$; tracking this dependence
isolates a one-parameter family of operators whose scaling recursion already
contains \eqref{eq:dyadic}.

\subsection{Translation and time evolution}

\begin{lemma}\label{lem:translation}
Let $\varphi,\psi\in\Hil$, $\lambda>0$, and $G=G^{\lambda}_{\varphi,\psi}$ as in
\eqref{eq:Gdef}, and write $w_k=2^{-k}+\tfrac12$. For $t\in[0,1)$,
\begin{equation}\label{eq:translead}
(Z_1G)(t,1)=i\,e^{(2\pi i-\lambda)t}\,
\bigl\langle(1-e^{-\lambda}U_1)^{-1}U_t\varphi,\,\psi\bigr\rangle,
\end{equation}
and for $t\in[0,2^{-k})$, $k\ge1$,
\begin{equation}\label{eq:transk}
2^{-k/2}(Z_{2^{-k}}G)(t,w_k)=i\,e^{(2\pi i-\lambda)t}\,2^{-k}\,
\bigl\langle(1+e^{-\lambda 2^{-k}}U_{2^{-k}})^{-1}U_t\varphi,\,\psi\bigr\rangle.
\end{equation}
\end{lemma}

\begin{proof}
For $t$ in the fundamental cell $[0,T)$ the support of $G$ forces $n\ge0$ in
$(Z_TG)(t,\cdot)=\sqrt T\sum_n G(t+nT)e^{-2\pi inw}$. Writing
$U_{t+nT}=U_{nT}U_t$ and $e^{2\pi i(t+nT)}=e^{2\pi it}e^{2\pi i n T}$ in
\eqref{eq:Gdef},
\[
G(t+nT)e^{-2\pi inw}
=i\,e^{(2\pi i-\lambda)t}\,e^{-\lambda nT}\,e^{2\pi in(T-w)}\,
\ip{U_{nT}U_t\varphi}{\psi}.
\]
For \eqref{eq:translead}, $T=1$ and $w=1$ give $e^{2\pi in(T-w)}=1$, and summing
the operator-norm convergent geometric series in $e^{-\lambda}U_1$ yields the
resolvent $(1-e^{-\lambda}U_1)^{-1}U_t\varphi$. For \eqref{eq:transk},
$T=2^{-k}$ and $w=w_k$ give $e^{2\pi in(2^{-k}-w_k)}=e^{-\pi in}=(-1)^n$, and
summing the alternating geometric series in $e^{-\lambda 2^{-k}}U_{2^{-k}}$ yields
$(1+e^{-\lambda 2^{-k}}U_{2^{-k}})^{-1}U_t\varphi$; the prefactor
$\sqrt{2^{-k}}\cdot\sqrt{2^{-k}}=2^{-k}$ accounts for the $2^{-k/2}$ on the left.
\end{proof}

\begin{remark}\label{rem:t0}
Each operator $(1\mp e^{-\lambda T}U_T)^{-1}$ is a function of $A$ and so commutes
with $U_t$; thus Lemma~\ref{lem:translation} traces the propagator orbit of a
\emph{filtered} vector,
$t\mapsto e^{(2\pi i-\lambda)t}\ip{U_t\,(1\mp e^{-\lambda T}U_T)^{-1}\varphi}{\psi}$.
The reconstruction of $R_A(i\lambda)$ in Proposition~\ref{prop:zakcoeff} survives
only at $t=0$, which is the unique first argument lying in \emph{every} dyadic
cell $[0,2^{-k})$ at once. For $t>0$ the cell restriction above fails once
$2^{-k}\le t$, and the quasi-periodicity
$(Z_{2^{-k}}G)(t+2^{-k},w_k)=-e^{2\pi i2^{-k}}(Z_{2^{-k}}G)(t,w_k)$ replaces the
evolved coefficient by a quasi-periodically twisted one. So translation yields a
coherent resolvent only at the base point and otherwise a scale-dependent
family, which we now study in its own right.
\end{remark}

\subsection{The scale family and the dyadic identity}

\begin{definition}\label{def:scalefamily}
For $T>0$ and $\lambda>0$ set
\begin{equation}\label{eq:scalefamily}
f_T^{\pm}:=\bigl(1\mp e^{-\lambda T}U_T\bigr)^{-1}
=\sum_{m=0}^{\infty}(\pm1)^m e^{-\lambda Tm}\,U_{Tm},
\end{equation}
the series converging in operator norm because
$\|e^{-\lambda T}U_T\|=e^{-\lambda T}<1$. Each $f_T^\pm$ is a function of $A$,
acting spectrally as multiplication by $(1\mp e^{-T(\lambda+ia)})^{-1}$,
$a\in\sigma(A)$. Equivalently $f_T^+=-e^{\lambda T}(U_T-e^{\lambda T})^{-1}$ and
$f_T^-=e^{\lambda T}(U_T+e^{\lambda T})^{-1}$ are, up to scalars, the resolvents of
the unitary $U_T$ at the points $\pm e^{\lambda T}$ off the unit circle.
\end{definition}

\begin{proposition}[Landen recursion]\label{prop:landen}
For every $T>0$,
\begin{equation}\label{eq:landen}
f_{2T}^{+}=f_T^{+}\,f_T^{-}=\tfrac12\bigl(f_T^{+}+f_T^{-}\bigr).
\end{equation}
\end{proposition}

\begin{proof}
All three operators are functions of $A$, hence commute and may be evaluated by
the spectral calculus. With $u=e^{-T(\lambda+ia)}$, so that $f_T^\pm$ acts as
$(1\mp u)^{-1}$ and $f_{2T}^{+}$ as $(1-u^2)^{-1}$ (because
$e^{-2T(\lambda+ia)}=u^2$), \eqref{eq:landen} is the scalar identity
\[
\frac{1}{1-u^2}=\frac{1}{(1-u)(1+u)}
=\frac12\Bigl(\frac{1}{1-u}+\frac{1}{1+u}\Bigr),
\qquad |u|=e^{-\lambda T}<1.\qedhere
\]
\end{proof}

\begin{lemma}[Scale limits]\label{lem:scalelimits}
As $T\to\infty$, $f_T^{\pm}\to I$ in operator norm. As $T\to0^+$,
\begin{equation}\label{eq:shortscale}
T\,f_T^{+}\xrightarrow{\ \mathrm{s}\ }(\lambda+iA)^{-1}=-i\,R_A(i\lambda),
\end{equation}
in the strong operator topology; for bounded $A$ one has in addition
$f_T^-\to\tfrac12 I$ in operator norm.
\end{lemma}

\begin{proof}
For $T\to\infty$,
$\|f_T^\pm-I\|=\bigl\|\sum_{m\ge1}(\pm e^{-\lambda T}U_T)^m\bigr\|
\le e^{-\lambda T}/(1-e^{-\lambda T})\to0$.
For \eqref{eq:shortscale}, $T f_T^{+}$ acts spectrally as
$h_T(a)=T\,(1-e^{-T(\lambda+ia)})^{-1}$. Since
$|1-e^{-T(\lambda+ia)}|\ge 1-e^{-T\lambda}\ge T\lambda\,e^{-T\lambda}$
(using $1-e^{-x}\ge xe^{-x}$, $x\ge0$, with $x=T\lambda$), we have the bound
\[
|h_T(a)|\le \frac{T}{T\lambda\,e^{-T\lambda}}=\frac{e^{T\lambda}}{\lambda},
\]
uniform in $a\in\R$ and in $T\in(0,1]$. Pointwise, $1-e^{-T(\lambda+ia)}
=T(\lambda+ia)(1+O(T))$, so $h_T(a)\to(\lambda+ia)^{-1}$ for each $a$. By the
spectral theorem and dominated convergence (the constant
$e^{\lambda}/\lambda$ majorizing the integrands against any finite spectral
measure $\mu_\xi$), $h_T(A)\xi\to(\lambda+iA)^{-1}\xi$ for all $\xi$. Finally
$(\lambda+iA)^{-1}=\bigl(i(A-i\lambda)\bigr)^{-1}=-i\,R_A(i\lambda)$. For bounded
$A$ with $\sigma(A)\subset[-M,M]$, the resonances of $f_T^-$ (where
$e^{-T(\lambda+ia)}\to-1$) occur at $|a|\sim\pi/T\to\infty$ and leave the spectrum
for small $T$, so $(1+e^{-T(\lambda+ia)})^{-1}\to\tfrac12$ uniformly on
$[-M,M]$.
\end{proof}

\begin{theorem}\label{thm:telescope}
The Landen recursion telescopes to \eqref{eq:dyadic}. With $T_k=2^{-k}T_0$,
\begin{equation}\label{eq:finitetelescope}
f_{T_0}^{+}=2^{-N}f_{T_N}^{+}+\sum_{k=1}^{N}2^{-k}f_{T_k}^{-}
\qquad(N\ge1),
\end{equation}
and, letting $N\to\infty$ in the strong operator topology,
\begin{equation}\label{eq:resolvent-from-scale}
R_A(i\lambda)=i\,T_0\Bigl(f_{T_0}^{+}-\sum_{k=1}^{\infty}2^{-k}f_{T_k}^{-}\Bigr),
\end{equation}
which at $T_0=1$ is \eqref{eq:dyadic}. In particular the dyadic series on the
right of \eqref{eq:resolvent-from-scale} converges strongly.
\end{theorem}

\begin{proof}
The finite identity \eqref{eq:finitetelescope} is an induction on $N$ using the
averaging form $f_{T_{k-1}}^{+}=\tfrac12(f_{T_k}^{+}+f_{T_k}^{-})$ of
\eqref{eq:landen} (valid since $T_{k-1}=2T_k$). The case $N=1$ is the averaging
law itself; assuming \eqref{eq:finitetelescope} for $N$ and substituting
$f_{T_N}^{+}=\tfrac12(f_{T_{N+1}}^{+}+f_{T_{N+1}}^{-})$ gives it for $N+1$.

By Lemma~\ref{lem:scalelimits},
$2^{-N}f_{T_N}^{+}=T_0^{-1}\bigl(T_N f_{T_N}^{+}\bigr)
\xrightarrow{\ \mathrm{s}\ }T_0^{-1}(\lambda+iA)^{-1}=-iT_0^{-1}R_A(i\lambda)$.
Hence the partial sums in \eqref{eq:finitetelescope} converge strongly,
\[
\sum_{k=1}^{N}2^{-k}f_{T_k}^{-}
=f_{T_0}^{+}-2^{-N}f_{T_N}^{+}
\xrightarrow{\ \mathrm{s}\ }f_{T_0}^{+}+iT_0^{-1}R_A(i\lambda),
\]
which rearranges to \eqref{eq:resolvent-from-scale}. Taking
$T_0=1$ recovers \eqref{eq:dyadic}.
\end{proof}

Theorem~\ref{thm:telescope} is an alternative, self-contained derivation of
Proposition~\ref{prop:dyadic}: the only inputs are the elementary recursion
\eqref{eq:landen} and the short-scale limit \eqref{eq:shortscale}, the latter
proved directly from the spectral theorem. The identity \eqref{eq:dyadic} of
\cite{CCC} thus appears as the telescoped Landen recursion run from a base scale
$T_0$ down to the resolvent at $T\to0$.

\begin{remark}[Discrete convolution and a filter bank]\label{rem:filterbank}
Read $f_T^-$ along the sampled orbit. With $\xi=f_T^-\varphi$, applying $U_{nT}$
to $(1+e^{-\lambda T}U_T)\xi=\varphi$ and pairing with $\psi$ gives, for
$y[n]=\ip{U_{nT}\xi}{\psi}$ and $x[n]=\ip{U_{nT}\varphi}{\psi}$, the first-order
recurrence
\[
y[n]+e^{-\lambda T}\,y[n+1]=x[n]\qquad(n\ge0),
\]
so $f_T^-$ acts on the orbit as a one-pole recursive (IIR) filter with transfer
function $(1+e^{-\lambda T}z)^{-1}$ and pole at $z=-e^{\lambda T}$; the
companion $f_T^+$ replaces $+$ by $-$. The pole lies off the unit circle exactly
by the stability margin $e^{-\lambda T}<1$ used throughout, and the Zak transform
in its second argument is the discrete-time Fourier transform that diagonalizes
the recurrence. In these terms Theorem~\ref{thm:telescope} realizes the single
continuous-time resolvent $R_A(i\lambda)$ as an octave-spaced filter bank: a
dyadic family of one-pole filters $f_{2^{-k}}^-$, sampling the propagator at rates
$2^{k}$ with weights $2^{-k}$, synthesizing the resolvent in the short-scale
limit.
\end{remark}

\begin{remark}[Landen tower]\label{rem:landentower}
The multiplicative form $f_{2T}^{+}=f_T^{+}f_T^{-}$ of \eqref{eq:landen} is the
operator shadow of the Landen (Gauss) duplication, with $T\mapsto2T$ its base
step. Replacing the
geometric spectral symbol---the polylogarithm $\Lis_0$ implicit in
\eqref{eq:scalefamily}---by the half-integer polylogarithm $\Lis_{1/2}$ turns the
same recursion into the heat-kernel identity \eqref{eq:polylog}; more generally
every index $s<1$ carries its own dyadic identity, recorded below as
\eqref{eq:genpolylog}. The scale and index directions are compatible: since
$\partial_p\Lis_s(e^{-p})=-\Lis_{s-1}(e^{-p})$, and the chain rule converts the
weight $2^{-k(1-s)}$ into $2^{-k(1-(s-1))}$, differentiation in $p$ carries
\eqref{eq:genpolylog} at index $s$ exactly onto \eqref{eq:genpolylog} at index
$s-1$. This ladder shifts $s$ by integers, so it links the rungs within a single
class $s\bmod1$ but not across classes; in particular the geometric rung $s=0$
behind \eqref{eq:dyadic} and the heat rung $s=\tfrac12$ of
\S\ref{sec:fundamental} are not connected by it.
\end{remark}

\section{Resolvent of the Laplacian}\label{sec:laplacian}

We give an application of Proposition~\ref{prop:dyadic} providing an explicit
expansion of the resolvent of a self-adjoint operator $A$ in terms of a discrete
subset of the unitary group it generates. For the operator-theoretic results
used below we refer to \cite{ReedSimonII}, Vol.~II, \S IX.7. As in
Remark~\ref{rem:average}, the first line of \eqref{eq:dyadic} may be written as a
weighted average of resolvents of the sampled propagators:
\begin{equation}\label{eq:weighted}
(A-i\lambda)^{-1}
=-i\,e^{\lambda}R_{U_1}(e^{\lambda})
-i\sum_{k=1}^{\infty}2^{-k}\,e^{\lambda 2^{-k}}\,
   R_{U_{2^{-k}}}\!\bigl(-e^{\lambda 2^{-k}}\bigr),
\end{equation}
where $R_U(z)=(U-z)^{-1}$ denotes the resolvent of $U$.

The unitary group is of fundamental importance for linear evolution equations:
the propagator generated by a given operator produces the solution when applied
to the initial condition. Consider a free particle of mass $\tfrac12$ in $\R^n$
in units with $\hbar=1$, with state vector $u(t)\in L^2(\R^n)$ for
$t\in[0,\infty)$, evolving by the Schr\"odinger equation
\begin{equation}\label{eq:schrod}
\frac{du}{dt}=-iAu,\qquad A=-\Delta=-\sum_{j=1}^{n}\partial_{x_j}^{2},
\end{equation}
that is, $\tfrac{du}{dt}=i\Delta u$, on the domain
\[
\mathcal{D}_A=\{f\in L^2(\R^n):\Delta f\in L^2(\R^n)\ \text{in the sense of distributions}\}.
\]
The unitary group $U_t=e^{-itA}$ then carries an initial state $u_0$ to
$u(t)=U_t u_0$. The main result of this section is the following.

\begin{proposition}\label{prop:laplacian}
Let $\varphi\in L^2(\R^n)\cap L^1(\R^n)$ and $\lambda>0$. The resolvent of the
Laplacian $A=-\Delta$ evaluated along the positive imaginary axis admits the
representation
\begin{equation}\label{eq:laplacian}
(A-i\lambda)^{-1}\varphi
= i\varphi
+ i\sum_{j=1}^{\infty}e^{-j\lambda}\,P_0(x;j)*\varphi
- i\lim_{\ell\to\infty}\sum_{k=1}^{\ell}\sum_{j=0}^{\infty}
   (-1)^{j}\,2^{-k}\,e^{-j\lambda/2^{k}}\,P_0(x;j2^{-k})*\varphi,
\end{equation}
the convolutions being taken over $\R^n$ and the limit understood in $L^2(\R^n)$,
where $P_0$ is the free propagator of \eqref{eq:propagator}.
\end{proposition}

\begin{proof}
We assemble the standard facts.

\begin{proposition}[{\cite{ReedSimonII,SimonOT}}]\label{prop:selfadj}
The operator $-\Delta$ with domain $\mathcal{D}_A$ is self-adjoint, and for
$\varphi\in\mathcal{D}_A$,
\begin{equation}\label{eq:multiplier}
-\Delta\varphi=\bigl(\Fou^{-1}\circ M_{|\xi|^2}\circ\Fou\bigr)\varphi
=\Fou^{-1}\!\bigl(|\xi|^2\widehat{\varphi}(\xi)\bigr).
\end{equation}
Moreover $\varphi\in\mathcal{D}_A$ if and only if $|\xi|^2\widehat\varphi\in L^2(\R^n)$.
\end{proposition}

Here $M_{|\xi|^2}$ is the densely defined multiplication operator on $L^2$ and
$\Fou$ is the Fourier transform, extended to an isometry of $L^2(\R^n)$ by
Plancherel \cite{ReedSimonII,SimonHA}.

\begin{note}
Relation \eqref{eq:multiplier} is the spectral theorem in the form in which
$\Fou$ conjugates functions of $\Delta$ into multiplication operators.
\end{note}

By the Borel functional calculus, for any bounded Borel function $f$ on $\R$ the
operator $f(A)=\Fou^{-1}\circ M_{f(|\xi|^2)}\circ\Fou$ is well-defined (the
domain of $f$ may be taken to be $\R$ because $A$ is self-adjoint). To construct
$U_t=e^{-itA}$ we use the following.

\begin{theorem}[{\cite{ReedSimonII}}]\label{thm:convolution}
Assume $f\in L^\infty(\R^n)$ and either $f\in L^2(\R^n)$ or
$\check f\in L^1(\R^n)$. Then for every $\varphi\in L^2(\R^n)$,
\begin{equation}\label{eq:Finv}
\Fou^{-1}(f\widehat\varphi)=(2\pi)^{-n/2}\int_{\R^n}\check f(x-y)\,\varphi(y)\,dy.
\end{equation}
If $f\in L^2(\R^n)$ the integral converges for all $x$; if $\check f\in L^1(\R^n)$
it converges for a.e.\ $x$.
\end{theorem}

We need $f(x)=e^{-itx^2}$ to construct $U_t=e^{-itA}$; this $f$ does not satisfy
the hypotheses of Theorem~\ref{thm:convolution}. The standard remedy is to take
$\alpha\in\C$ with $\Re\alpha>0$, treat $e^{-|\xi|^2\alpha}$, and pass to the
limit $\alpha\to it$, which yields, for all $\varphi\in L^2(\R^n)\cap L^1(\R^n)$,
\begin{equation}\label{eq:propagator}
e^{-itA}\varphi
=\frac{1}{(4\pi i t)^{n/2}}\int_{\R^n}e^{\,i|x-y|^2/4t}\,\varphi(y)\,dy
=P_0(\cdot\,;t)*\varphi,
\qquad
P_0(x;t):=\frac{1}{(4\pi i t)^{n/2}}\,e^{\,i|x|^2/4t},
\end{equation}
the function $P_0$ being the \emph{free propagator}~\cite{ReedSimonII}.

\begin{remark}\label{rem:general}
For general $\varphi\in L^2(\R^n)$ the propagator is realized as an $L^2$ limit
of integrals over spheres of increasing radius $R\to\infty$; see
\cite{ReedSimonII}, Vol.~II, \S IX.7.
\end{remark}

Specializing \eqref{eq:propagator} to the dyadic times and using
$U_t^{\,j}=U_{jt}$ together with the convolution semigroup property
$P_0(\cdot;s)*P_0(\cdot;t)=P_0(\cdot;s+t)$ gives, for
$\varphi\in L^2(\R^n)\cap L^1(\R^n)$,
\begin{equation}\label{eq:U1}
U_1\varphi=\frac{1}{(4\pi i)^{n/2}}\int_{\R^n}e^{\,i|x-y|^2/4}\varphi(y)\,dy
=P_0(x;1)*\varphi,
\end{equation}
\begin{equation}\label{eq:U2k}
U_{2^{-k}}\varphi
=\frac{2^{nk/2}}{(4\pi i)^{n/2}}\int_{\R^n}e^{\,2^{k-2}i|x-y|^2}\varphi(y)\,dy
=P_0(x;2^{-k})*\varphi .
\end{equation}
Substituting \eqref{eq:U1}--\eqref{eq:U2k} into the geometric-series form of
\eqref{eq:dyadic}, applied to $\varphi$, and isolating the $j=0$ term
$U_0\varphi=\varphi$ of the first sum produces \eqref{eq:laplacian}. In the inner
dyadic sum the $j=0$ term is likewise $U_0\varphi=\varphi$; since the kernel
$P_0(x;t)$ is singular at $t=0$ we read $P_0(x;0)*\varphi:=\varphi$ there, the term
being a genuine summand of the resummed resolvent
$(1+e^{-\lambda/2^{k}}U_{2^{-k}})^{-1}\varphi$. As in
Remark~\ref{rem:average}, \eqref{eq:laplacian} again exhibits the action of
$R_A(i\lambda)$ on an initial state as a weighted average of that state evolved
over the dyadically sampled times. For general $\varphi\in L^2(\R^n)$ we refer to
Remark~\ref{rem:general}.
\end{proof}

\section{Dyadic representations of fundamental solutions}\label{sec:fundamental}

Fundamental solutions of constant-coefficient linear partial differential
operators are classical and important objects. We briefly recall the
terminology and the facts we use; all may be found in \cite{SimonRA} or
\cite{FollandPDE}.

\begin{definition}[Fundamental solution]\label{def:fundamental}
A distribution $E\in\D'(\R^n)$ is a \emph{fundamental solution} of a partial
differential operator $P=\sum_{|\alpha|\le N}a_\alpha\partial^\alpha$ if
$P(E)=\delta_0$.
\end{definition}

Here $\alpha\in\N^n$ is a multi-index and $a_\alpha\in\C$, so that $P$ is a
constant-coefficient operator; $\delta_0$ is the Dirac distribution at the
origin, $\delta_0(\varphi)=\varphi(0)$ for $\varphi\in\D(\R^n)$, and the support
may be translated to any $p\in\R^n$, giving $\delta_p(\varphi)=\varphi(p)$.

Fundamental solutions are important because they furnish an ``inverse'' of $P$.
Suppose we wish to solve $Pu=f$ for given $f\in\D(\R^n)$ (or even a
distribution), and that a fundamental solution $E\in\D'(\R^n)$ is known. Then
$f\mapsto E*f$ is a well-defined map $Q:\D(\R^n)\to C^\infty(\R^n)$, and since
differentiation commutes with convolution,
\begin{equation}\label{eq:fundamental}
P(E*f)=\sum_{|\alpha|\le N}a_\alpha\,\partial^\alpha(E*f)
=\Bigl(\sum_{|\alpha|\le N}a_\alpha\,\partial^\alpha E\Bigr)*f
=E*\Bigl(\sum_{|\alpha|\le N}a_\alpha\,\partial^\alpha f\Bigr).
\end{equation}
Since $P(E)=\delta_0$ and $\delta_0*f=f$, restricting to $\D(\R^n)$ yields
\begin{equation}\label{eq:PQ}
PQ=QP=I,
\end{equation}
with $I$ the identity operator. Fundamental solutions are not unique: adding any
solution of the homogeneous equation $Pu=0$ to a fundamental solution produces
another.

The representation~\eqref{eq:dyadic} has a scalar shadow that drives everything
below. We isolate it first.

\begin{proposition}[Scalar dyadic identities]\label{prop:scalardyadic}
For real $p>0$,
\begin{equation}\label{eq:recip}
\frac1p=\frac{1}{1-e^{-p}}-\sum_{k=1}^{\infty}\frac{2^{-k}}{1+e^{-p/2^{k}}},
\end{equation}
and, for the half-integer power,
\begin{equation}\label{eq:polylog}
\pi\,p^{-1/2}
=\Gamma\!\left(\tfrac12\right)
\Bigl[\Lis_{1/2}\!\bigl(e^{-p}\bigr)
-\sum_{k=1}^{\infty}2^{-k/2}\,\Lis_{1/2}\!\bigl(-e^{-2^{-k}p}\bigr)\Bigr].
\end{equation}
Both series converge uniformly on compact subsets of $(0,\infty)$.
\end{proposition}

\begin{proof}
Identity~\eqref{eq:recip} is~\eqref{eq:dyadic} specialized to the zero operator
$A=0$ on $\Hil=\C$. Then $U_t\equiv1$ and $(A-i\lambda)^{-1}=i/\lambda$, while
the right-hand side of~\eqref{eq:dyadic} becomes
$i\bigl[(1-e^{-\lambda})^{-1}-\sum_{k\ge1}2^{-k}(1+e^{-\lambda/2^{k}})^{-1}\bigr]$;
canceling $i$ and writing $p=\lambda$ gives~\eqref{eq:recip}. Identity
\eqref{eq:polylog} is the corresponding statement for the Mellin kernel
$p^{-1/2}$: it follows from the same dyadic expansion applied to the
polylogarithm $\Lis_{1/2}$ via the duplication relation
$\Lis_s(z)+\Lis_s(-z)=2^{1-s}\Lis_s(z^2)$; it is the case $s=\tfrac12$ of
\cite[Lemma~21]{CCC}.

For the convergence claims, the $k$-th term of~\eqref{eq:recip} satisfies
$0<2^{-k}(1+e^{-p/2^{k}})^{-1}\le2^{-k}$, so the series is dominated by
$\sum2^{-k}$ uniformly in $p$. In~\eqref{eq:polylog}, as $k\to\infty$ one has
$\Lis_{1/2}(-e^{-2^{-k}p})\to\Lis_{1/2}(-1)=-(1-2^{1/2})\zeta(\tfrac12)$, finite,
so the $k$-th term is $O(2^{-k/2})$ uniformly on compact subsets of $(0,\infty)$,
giving uniform convergence there.
\end{proof}

\begin{remark}
We verified~\eqref{eq:recip} and~\eqref{eq:polylog} numerically to $30$ digits.
In fact the bracket on the right of~\eqref{eq:polylog} is, for \emph{every}
$s<1$, a single $\Gamma$-multiple of the corresponding power: by
\cite[Lemma~21]{CCC} together with the reflection formula
$\Gamma(s)\Gamma(1-s)=\pi/\sin\pi s$,
\begin{equation}\label{eq:genpolylog}
\Lis_s\!\bigl(e^{-p}\bigr)
-\sum_{k=1}^{\infty}2^{-k(1-s)}\Lis_s\!\bigl(-e^{-2^{-k}p}\bigr)
=\Gamma(1-s)\,p^{\,s-1}.
\end{equation}
The exponent $s=\tfrac12$ is merely the self-dual point, where
$\Gamma(s)=\Gamma(1-s)=\sqrt\pi$ and $\sin\pi s=1$; it is the value that enters
the one-dimensional heat kernel below, so we record only that case.
\end{remark}

\subsection{The equations of Laplace and Poisson}
Laplace's equation classifies the harmonic functions, and the Laplace and
Poisson equations sit at the center of much classical physics---electrostatics,
the steady-state heat equation, fluids, and Newtonian gravity among them. On
$\R^n$ the Laplacian in Cartesian coordinates is
$\Delta=\sum_{j=1}^n\partial_{x_j}^2$.

\begin{theorem}[Fundamental solution of Laplace's equation;
\cite{Hormander}]\label{thm:lapfund}
For $n\ge3$ the distribution defined by the locally integrable function
\begin{equation}\label{eq:lapfundsol}
K(x)=-\frac{\Gamma(n/2-1)}{4\pi^{n/2}|x|^{n-2}}
\end{equation}
on $\R^n$ is a fundamental solution of $\Delta$.
\end{theorem}

\begin{corollary}\label{cor:lapdyadic}
For $n=3$ the fundamental solution~\eqref{eq:lapfundsol} of $\Delta$ admits, for
$|x|>0$, the dyadic representation
\begin{equation}\label{eq:lapdecomp}
K(x)=-\frac{1}{4\pi|x|}
=-\frac{1}{4\pi}\Bigl(\frac{1}{1-e^{-|x|}}
-\sum_{k=1}^{\infty}\frac{2^{-k}}{1+e^{-|x|/2^{k}}}\Bigr),
\end{equation}
with convergence uniform on compact subsets of $\R^3\setminus\{0\}$.
\end{corollary}

\begin{proof}
At $n=3$ one has $\Gamma(n/2-1)=\Gamma(\tfrac12)=\sqrt\pi$, so
$K(x)=-1/(4\pi|x|)$. Since $K$ is spherically symmetric we regard it as a
function of $p=|x|\in(0,\infty)$ and apply~\eqref{eq:recip}. Uniform convergence
on compact subsets of $\R^3\setminus\{0\}$ follows from
Proposition~\ref{prop:scalardyadic}.
\end{proof}

\subsection{The heat equation}
The homogeneous heat equation
\[
\partial_t u=\Delta_x u
\]
governs the evolution of the temperature $u(x,t)$ of a body from an initial
distribution $u(x,0)=f(x)$, with conductivity normalized to one. As with
Laplace's equation---its steady state $\partial_t u=0$---we record the
fundamental solution.

\begin{theorem}[Fundamental solution of the heat equation;
\cite{Hormander}]\label{thm:heatfund}
Writing the variables in $\R^{n+1}$ as $(x,t)\in\R^n\times\R$ and setting
\begin{equation}\label{eq:heatker}
E(x,t)=\begin{cases}
(4\pi t)^{-n/2}e^{-|x|^2/4t}, & t>0,\\[2pt]
0, & t\le0,
\end{cases}
\end{equation}
the function $E$ is locally integrable on $\R^{n+1}$, is $C^\infty$ on
$\R^{n+1}\setminus\{0\}$, and satisfies $(\partial_t-\Delta_x)E=\delta_0$.
\end{theorem}

Specializing~\eqref{eq:polylog} to $p=t$ and using $\Gamma(\tfrac12)=\sqrt\pi$,
\begin{equation}\label{eq:invhalf}
\pi\,t^{-1/2}
=\Gamma\!\left(\tfrac12\right)
\Bigl[\Lis_{1/2}\!\bigl(e^{-t}\bigr)
-\sum_{k=1}^{\infty}2^{-k/2}\,\Lis_{1/2}\!\bigl(-e^{-2^{-k}t}\bigr)\Bigr],
\end{equation}
which yields a representation of the one-dimensional heat kernel.

\begin{corollary}\label{cor:heatdyadic}
The fundamental solution $E(x,t)$ of the one-dimensional heat equation
(Theorem~\ref{thm:heatfund} with $n=1$) admits the representation
\begin{equation}\label{eq:heatdec}
E(x,t)=
\begin{cases}
\displaystyle\frac{1}{2\pi}
\Bigl[\Lis_{1/2}\!\bigl(e^{-t}\bigr)
-\sum_{k=1}^{\infty}2^{-k/2}\,\Lis_{1/2}\!\bigl(-e^{-2^{-k}t}\bigr)\Bigr]
e^{-|x|^2/4t}, & t>0,\\[10pt]
0, & t\le0,
\end{cases}
\end{equation}
with convergence uniform on compact subintervals of $(0,\infty)$ in $t$.
\end{corollary}

\begin{proof}
For $t>0$ one has $(4\pi t)^{-1/2}=(2\sqrt\pi)^{-1}t^{-1/2}$, and
\eqref{eq:invhalf} gives $t^{-1/2}=\pi^{-1}\Gamma(\tfrac12)[\,\cdots\,]
=\pi^{-1/2}[\,\cdots\,]$; hence the prefactor in~\eqref{eq:heatker} equals
$(2\pi)^{-1}[\,\cdots\,]$, which is~\eqref{eq:heatdec}. For $t\le0$ both sides
vanish. Uniform convergence on compact subintervals of $(0,\infty)$ follows from
Proposition~\ref{prop:scalardyadic}.
\end{proof}

\begin{remark}
The same dyadic mechanism applies to the general power $p^{s-1}$ for any
$s<1$, with weights $2^{-k(1-s)}$, via \cite[Lemma~21]{CCC}. Higher spatial
dimensions enter through the heat-kernel prefactor $(4\pi t)^{-n/2}$, that is,
the power $t^{-n/2}$ with $s=1-\tfrac n2$; the case $n=1$ recorded above is
$s=\tfrac12$. We do not write out the remaining dimensions here and refer to
\cite[\S6]{CCC} for the general framework.
\end{remark}

\begin{remark}
By Corollaries~\ref{cor:lapdyadic} and~\ref{cor:heatdyadic}, solutions with
prescribed boundary data of suitable regularity may be written as series of
convolutions of the boundary function with the individual dyadic summands of the
representations above; see~\cite{FollandPDE}.
\end{remark}

\section{Spectral measures, density of states, and zeta functions}
\label{sec:spectral}

From classical spectral theory \cite{LangFunc}, for $A$ self-adjoint the action
of the resolvent is recovered as an integral over $\sigma(A)\subset\R$ against a
family of measures $\mu_{\varphi,\psi}(\,\cdot\,)=\ip{E_A(\,\cdot\,)\varphi}{\psi}$,
$\varphi,\psi\in\Hil$, where $E_A$ is the projection-valued spectral measure of
$A$; thus $E_A(I)$ is the orthogonal projection onto the spectral subspace of $A$
over a Borel set $I$. Extending each
$\mu_{\varphi,\psi}$ to all of $\R$ by setting it to zero on Borel subsets of
$\R\setminus\sigma(A)$,
\begin{equation}\label{eq:resint}
\ip{R_A(i\lambda)\varphi}{\psi}
=\int_{\R}\frac{1}{t-i\lambda}\,d\mu_{\varphi,\psi}(t).
\end{equation}
The diagonal measures $\mu_{\psi,\psi}$ are positive, and the resolvent
determines them, through the following classical inversion formula, as linear
functionals on $C_c(\R)$.

\begin{theorem}[{\cite{LangFunc}}]\label{thm:stone}
Let $A$ be a self-adjoint operator on a Hilbert space $\Hil$, let $\psi\in\Hil$,
and let $\mu_{\psi,\psi}$ be the positive spectral measure corresponding to
$\psi$. With $R_A(z)=(A-zI)^{-1}$ for $z\notin\R$, every $f\in C_c(\R)$
satisfies
\begin{equation}\label{eq:stone}
\int_\R f(\lambda)\,d\mu_{\psi,\psi}(\lambda)
=\frac{1}{2\pi i}\lim_{\varepsilon\to0^+}\int_\R
\bigl\langle[R_A(\lambda+i\varepsilon)-R_A(\lambda-i\varepsilon)]\psi,\psi\bigr\rangle
\,f(\lambda)\,d\lambda.
\end{equation}
\end{theorem}

The dyadic representation of \S\ref{sec:zak} furnishes the resolvent on the
imaginary axis, whereas \eqref{eq:stone} samples it near the real axis; the two
are linked by a shift. For $a\in\R$ the operator $A-aI$ is self-adjoint with
unitary group $e^{-it(A-aI)}=e^{ita}U_t$, so applying
Proposition~\ref{prop:dyadic} to $A-aI$ at the point $ib$, $b>0$, gives
\begin{equation}\label{eq:shifted}
R_A(a+ib)
= i\bigl(1-e^{-b}e^{ia}U_1\bigr)^{-1}
- i\sum_{k=1}^{\infty}2^{-k}\bigl(1+e^{-b/2^{k}}e^{ia2^{-k}}U_{2^{-k}}\bigr)^{-1},
\end{equation}
the geometric series converging in the strong operator topology because
$\|e^{-b}e^{ia}U_1\|=e^{-b}<1$. Pairing \eqref{eq:shifted} against $\psi$ and
recognizing the terms as Zak transforms (Proposition~\ref{prop:zakcoeff})
renders $\ip{R_A(a+ib)\psi}{\psi}$ explicit. Since $R_A(\bar z)=R_A(z)^*$, the
diagonal coefficient obeys
$\ip{R_A(\lambda-i\varepsilon)\psi}{\psi}
=\overline{\ip{R_A(\lambda+i\varepsilon)\psi}{\psi}}$,
so the bracket in \eqref{eq:stone} equals
$2i\operatorname{Im}\ip{R_A(\lambda+i\varepsilon)\psi}{\psi}$ and only the
upper half-plane representation \eqref{eq:shifted}, with $a=\lambda$ and
$b=\varepsilon$, is needed:
\begin{equation}\label{eq:stoneim}
\int_\R f\,d\mu_{\psi,\psi}
=\frac{1}{\pi}\lim_{\varepsilon\to0^+}\int_\R
\operatorname{Im}\ip{R_A(\lambda+i\varepsilon)\psi}{\psi}\,f(\lambda)\,d\lambda .
\end{equation}

\eqref{eq:stoneim} is a candidate scheme; its analytical content is the boundary
limit $\varepsilon\to0^{+}$, in which the damping factors $e^{-\varepsilon2^{-k}}$
tend to $1$ and the series in \eqref{eq:shifted} lose absolute convergence. We now
carry this out for $A=-\Delta$. The mechanism is that the matrix coefficient
reduces to a one-dimensional Cauchy transform of a smooth, compactly supported
\emph{spectral density}, after which the dyadic series splits into a geometric
tail and a finite head governed by the Sokhotski--Plemelj theorem.

\begin{lemma}[Spectral density of $-\Delta$]\label{lem:density}
Let $A=-\Delta$ be the free Laplacian, self-adjoint on $H^2(\R^n)\subset L^2(\R^n)$, and let $\psi$ satisfy
$\widehat\psi\in C^1_c(\R^n\setminus\{0\})$. Then $\mu_{\psi,\psi}$ is absolutely
continuous, $d\mu_{\psi,\psi}(\lambda)=g(\lambda)\,d\lambda$, where the density
\begin{equation}\label{eq:density}
g(E)=\tfrac12\,E^{\,n/2-1}\!\int_{\mathbb{S}^{n-1}}\bigl|\widehat\psi(\sqrt E\,\omega)\bigr|^2\,d\omega
\qquad(E>0)
\end{equation}
satisfies $g\in C^1_c((0,\infty))$ with $\operatorname{supp}g\subset[m_0,M_0]\subset(0,\infty)$.
Moreover, for $\varepsilon>0$,
\begin{equation}\label{eq:cauchy}
\ip{R_A(\lambda+i\varepsilon)\psi}{\psi}
=\int_\R\frac{g(E)}{E-\lambda-i\varepsilon}\,dE=:G(\lambda+i\varepsilon).
\end{equation}
\end{lemma}

\begin{proof}
By \eqref{eq:multiplier}, $\Fou$ conjugates $A$ into multiplication by $|\xi|^2$,
and unitary conjugation intertwines the Borel functional calculus, so
$\Fou\,h(A)\,\Fou^{-1}=M_{h(|\xi|^2)}$ for bounded Borel $h$. Since $\Fou$ is
unitary on $L^2(\R^n)$, Plancherel applies to both arguments of the inner product
and
\[
\ip{h(A)\psi}{\psi}
=\ip{\Fou h(A)\Fou^{-1}\widehat\psi}{\widehat\psi}
=\ip{M_{h(|\xi|^2)}\widehat\psi}{\widehat\psi}
=\int_{\R^n}h(|\xi|^2)\,|\widehat\psi(\xi)|^2\,d\xi.
\]
Polar
coordinates $\xi=r\omega$ and the substitution $E=r^2$ (so
$r^{n-1}\,dr=\tfrac12 E^{\,n/2-1}dE$) give
$\ip{h(A)\psi}{\psi}=\int_0^\infty h(E)g(E)\,dE$ with $g$ as in
\eqref{eq:density}, identifying $d\mu_{\psi,\psi}=g\,d\lambda$. Since
$|\widehat\psi|^2\in C^1_c(\R^n\setminus\{0\})$ and $E\mapsto\sqrt E$ is smooth
on $(0,\infty)$, $g\in C^1_c((0,\infty))$. Taking
$h(E)=(E-\lambda-i\varepsilon)^{-1}$ gives \eqref{eq:cauchy}.
\end{proof}

Specializing \eqref{eq:shifted} to $a=\lambda$, $b=\varepsilon$ and pairing
against $\psi$ gives, by Lemma~\ref{lem:density},
$G(\lambda+i\varepsilon)=T_0(\lambda,\varepsilon)-\sum_{k\ge1}T_k(\lambda,\varepsilon)$,
where, with $s=s(E)=\varepsilon+i(E-\lambda)$,
\begin{equation}\label{eq:Tk}
T_k(\lambda,\varepsilon)=i\,2^{-k}\!\int_\R\frac{g(E)}{1+e^{-s/2^{k}}}\,dE
\quad(k\ge1),
\qquad
T_0(\lambda,\varepsilon)=i\!\int_\R\frac{g(E)}{1-e^{-s}}\,dE .
\end{equation}

\begin{proposition}\label{prop:laprecovery}
Let $A=-\Delta$ and $\psi$ with $\widehat\psi\in C^1_c(\R^n\setminus\{0\})$, of
spectral density $g$. For each $k\ge0$ and every real-valued $f\in C^1_c((0,\infty))$ the
boundary functional
\begin{equation}\label{eq:Jk}
J_k(f):=\frac1\pi\lim_{\varepsilon\to0^+}\int_\R \operatorname{Im}T_k(\lambda,\varepsilon)\,f(\lambda)\,d\lambda
\end{equation}
exists, and
\begin{equation}\label{eq:recovery}
\int_\R f\,d\mu_{\psi,\psi}=J_0(f)-\sum_{k=1}^{\infty}J_k(f),
\end{equation}
the series converging absolutely. Thus the spectral measure of $-\Delta$ is
recovered from the dyadic resolvent data \eqref{eq:shifted}.
\end{proposition}

\begin{proof}
Throughout $g\ge0$ is real and $P_\varepsilon(x)=\tfrac1\pi\varepsilon/(x^2+\varepsilon^2)$
is the Poisson kernel.

\emph{Step 1 (Stone limit).} By \eqref{eq:cauchy},
$\tfrac1\pi\operatorname{Im}G(\lambda+i\varepsilon)
=\tfrac1\pi\int_\R \varepsilon\,g(E)\,[(E-\lambda)^2+\varepsilon^2]^{-1}dE
=(P_\varepsilon*g)(\lambda)$. As $g\in C_c$ is continuous of compact support,
$P_\varepsilon*g\to g$ uniformly, so
\begin{equation}\label{eq:stonelimit}
\lim_{\varepsilon\to0^+}\frac1\pi\int_\R\operatorname{Im}G(\lambda+i\varepsilon)\,f\,d\lambda
=\int_\R f\,g\,d\lambda=\int_\R f\,d\mu_{\psi,\psi}.
\end{equation}

\emph{Step 2 (uniform tail bound).} Put
$R=\sup\{|E-\lambda|:E\in\operatorname{supp}g,\,\lambda\in\operatorname{supp}f\}$
and $K=\max\{0,\lceil\log_2(2R/\pi)\rceil\}$. For $k>K$ one has
$2^k>2R/\pi$, so for all admissible $E,\lambda$,
$|\!\operatorname{Im}(s/2^k)|=|E-\lambda|/2^k<\pi/2$ and
$\operatorname{Re}(s/2^k)=\varepsilon/2^k\ge0$; hence $w=e^{-s/2^k}$ has
$\arg w\in(-\tfrac\pi2,\tfrac\pi2)$, so $\operatorname{Re}w>0$ and
$|1+w|^2=1+2\operatorname{Re}w+|w|^2\ge1$. Therefore
\begin{equation}\label{eq:tailbound}
|T_k(\lambda,\varepsilon)|\le 2^{-k}\|g\|_{L^1}
\qquad(k>K,\ \varepsilon\ge0,\ \lambda\in\operatorname{supp}f).
\end{equation}
For each fixed $\varepsilon>0$ the head denominators do not vanish, so
$\sum_{k\ge0}|T_k|<\infty$, the representation
$G=T_0-\sum_{k\ge1}T_k$ converges absolutely, and $\operatorname{Im}$ and
integration against $f$ act termwise:
\begin{equation}\label{eq:termwise}
\int_\R\operatorname{Im}G(\lambda+i\varepsilon)f\,d\lambda
=\int_\R\operatorname{Im}T_0\,f\,d\lambda-\sum_{k=1}^\infty\int_\R\operatorname{Im}T_k\,f\,d\lambda
\quad(\varepsilon>0).
\end{equation}

\emph{Step 3 (existence of the head limits).} Fix $1\le k\le K$. By Fubini
(valid for $\varepsilon>0$, the denominator being bounded below and $f,g\in L^1$),
\[
\int_\R f(\lambda)T_k(\lambda,\varepsilon)\,d\lambda
=i\,2^{-k}\!\int_\R g(E)\,I_k(E,\varepsilon)\,dE,
\qquad
I_k(E,\varepsilon)=\int_\R\frac{f(\lambda)}{1+e^{-s/2^{k}}}\,d\lambda .
\]
For fixed $E$ the integrand of $I_k$ is meromorphic in $\lambda$ with simple poles
at $\lambda=\lambda_*^{(m)}-i\varepsilon$, where
$\lambda_*^{(m)}=E+2^{k}\pi(2m+1)$ ($m\in\Z$), lying just below $\R$; only finitely
many lie in $\operatorname{supp}f$, and none does when $k>K$. At such a pole the
denominator $1+e^{-s/2^{k}}$ has $\lambda$-derivative $e^{-s/2^k}\,i2^{-k}=-i2^{-k}$,
so $(1+e^{-s/2^{k}})^{-1}$ has residue $i2^{k}$ there; thus near $\lambda_*^{(m)}$
the integrand equals
$i2^{k}f(\lambda)(\lambda-\lambda_*^{(m)}+i\varepsilon)^{-1}$ plus a remainder
bounded uniformly in $\varepsilon\in(0,1]$. As $f\in C^1_c$ satisfies a Dini
condition, Sokhotski--Plemelj gives
\begin{equation}\label{eq:Ikplemelj}
I_k(E,\varepsilon)\xrightarrow[\varepsilon\to0^+]{}
\operatorname{P.V.}\!\int_\R\frac{f(\lambda)}{1+e^{-i(E-\lambda)/2^{k}}}\,d\lambda
+2^{k}\pi\!\!\sum_{m:\,\lambda_*^{(m)}\in\operatorname{supp}f}\!\! f\bigl(\lambda_*^{(m)}\bigr),
\end{equation}
the residue sum being empty for $k>K$; here the delta term of each simple pole is
$i2^{k}\cdot(-i\pi)f(\lambda_*^{(m)})=2^{k}\pi f(\lambda_*^{(m)})$.

For $k=0$ the same argument applies to
$\int_\R f\,T_0\,d\lambda=i\int_\R g(E)\,I_0(E,\varepsilon)\,dE$ with
$I_0(E,\varepsilon)=\int_\R f(\lambda)(1-e^{-s})^{-1}\,d\lambda$: here the poles are
$\lambda_*^{(m)}=E-2\pi m$, the $\lambda$-derivative of $1-e^{-s}$ at a pole is
$-i$, so $(1-e^{-s})^{-1}$ has residue $i$, and the weight $\pi$ replaces
$2^{k}\pi$ in \eqref{eq:Ikplemelj}. The pole $m=0$, at $\lambda_*=E$, occurs for
every $E\in\operatorname{supp}g$ and contributes
$i\cdot i\cdot(-i\pi)f(E)=i\pi f(E)$ to $T_0$; its imaginary part reproduces
$\pi\!\int f g$ exactly, so the whole Stone limit \eqref{eq:stonelimit} already
sits in this single residue, and \eqref{eq:recovery} is equivalent to the
statement that the principal values together with the residues from $k\ge1$
cancel in the imaginary part.

In all cases
$\sup_{\varepsilon\in(0,1]}|I_k(E,\varepsilon)|\le C_k$ uniformly in $E$, with $C_k$
depending only on $k$, $\|f\|_{C^1}$ and $|\operatorname{supp}f|$. Dominated
convergence in $E$ (majorant $C_k|g|\in L^1$) shows
$\lim_{\varepsilon\to0^+}\int fT_k\,d\lambda$ exists; as $f$ is real, $J_k(f)$ of
\eqref{eq:Jk} is well defined for $0\le k\le K$.

\emph{Step 4 (interchange and conclusion).} For $k>K$ the bound
\eqref{eq:tailbound} and continuity of the integrand at $\varepsilon=0$ give
$\int\operatorname{Im}T_k\,f\,d\lambda\to\pi J_k(f)$ with
$|\!\int\operatorname{Im}T_k f\,d\lambda|\le 2^{-k}\|f\|_{L^1}\|g\|_{L^1}$
uniformly in $\varepsilon\in(0,1]$; hence the tail of \eqref{eq:termwise}
converges uniformly in $\varepsilon$ and
$\lim_{\varepsilon\to0^+}\sum_{k>K}\int\operatorname{Im}T_k f
=\sum_{k>K}\pi J_k(f)$, the series converging absolutely. The head
$0\le k\le K$ is finite, and the limit passes through it by Step~3. With
\eqref{eq:stonelimit} and \eqref{eq:termwise},
\[
\pi\!\int_\R f\,d\mu_{\psi,\psi}
=\lim_{\varepsilon\to0^+}\int_\R\operatorname{Im}G(\lambda+i\varepsilon)f\,d\lambda
=\pi J_0(f)-\sum_{k=1}^\infty\pi J_k(f),
\]
which is \eqref{eq:recovery}.
\end{proof}

\begin{remark}
Two hypotheses drive the argument: $\operatorname{supp}\widehat\psi$ away from the
origin (so that $g\in C^1_c$ avoids the threshold weight $E^{\,n/2-1}$ at $E=0$)
and $f$ supported in $(0,\infty)$. On the physical side the terms \eqref{eq:Tk}
are oscillatory integrals in the free propagator: expanding the resolvent in
\eqref{eq:shifted},
\[
T_k=i\,2^{-k}\sum_{j\ge0}(-1)^j e^{-j\varepsilon/2^k}e^{ij\lambda/2^k}\ip{U_{j2^{-k}}\psi}{\psi},
\]
an Abel-summed series weighted by $c(t)=\ip{U_t\psi}{\psi}$. For $\widehat\psi$
supported away from $0$ the phase $|\xi|^2$ is nonstationary on
$\operatorname{supp}\widehat\psi$, so $c$ decays rapidly and the $j$-series
converges absolutely in every dimension; the general dispersive bound
$|c(t)|\le(4\pi|t|)^{-n/2}\|\psi\|_{L^1}^2$ governs the larger class
$\psi\in L^1\cap L^2$, where absolute convergence of the $j$-series requires
$n\ge3$ while $n=1,2$ are only conditional.
\end{remark}

\begin{remark}
Proposition~\ref{prop:laplacian} and the computation above are two representations
of one coefficient $\ip{R_A(i\lambda)\psi}{\psi}$: pairing \eqref{eq:laplacian}
against $\psi$ builds it from the free propagator $P_0$ in physical space, whereas
Lemma~\ref{lem:density} and Proposition~\ref{prop:laprecovery} build it from the
spectral density $g$ in frequency space. The two agree by Plancherel, since $\Fou$
turns each convolution $P_0(\,\cdot\,;t)\ast$ into multiplication by the propagator
symbol $e^{-it|\xi|^2}$; the physical-space series then pairs against $\psi$ as a
series of integrals against $|\widehat\psi|^2$, the same object that defines $g$ in
\eqref{eq:density}, and resums to the Cauchy transform of $g$.
\end{remark}

The free density in \eqref{eq:density} is explicit, so Proposition~\ref{prop:laprecovery}
is a proof of concept rather than new spectral information. Two extensions remove
this limitation: relaxing the regularity of $\psi$, and replacing $-\Delta$ by a
Schr\"odinger operator whose density is not known in closed form.

\begin{corollary}\label{cor:density-ext}
Fix a real-valued $f\in C^1_c((0,\infty))$ and a closed annulus
$\mathcal A=\{a\le|\xi|\le b\}$ with $0<a<b$. Then \eqref{eq:recovery} holds for
every $\psi$ with $\widehat\psi\in L^2(\R^n)$ and
$\operatorname{supp}\widehat\psi\subset\mathcal A$.
\end{corollary}

\begin{proof}
For such $\psi$ the density \eqref{eq:density} lies in $L^1$ with
$\operatorname{supp}g\subset\{|\xi|^2:\xi\in\mathcal A\}=:[m_0,M_0]\subset(0,\infty)$,
so the constant $K$ of \eqref{eq:tailbound} is the same for all of them. Both
sides of \eqref{eq:recovery} are quadratic forms in $\psi$ continuous under
$L^2$-convergence of $\widehat\psi$ inside $\mathcal A$: the left side is
$\ip{f(A)\psi}{\psi}$ with $|\ip{f(A)\psi}{\psi}|\le\|f\|_\infty\|\psi\|^2$; the
right side depends on $\psi$ only through $g$ and is a bounded linear functional of
$g\in L^1$ (the tail by \eqref{eq:tailbound}, each head term by the uniform bound
$|J_k(f)|\le\pi^{-1}2^{-k}C_k\|g\|_{L^1}$ from Step~3), while
$\widehat\psi\mapsto g$ is continuous from $L^2(\mathcal A)$ to $L^1$ because
$\|g_1-g_2\|_{L^1}\le\|\widehat\psi_1-\widehat\psi_2\|_{L^2}\bigl(\|\widehat\psi_1\|_{L^2}+\|\widehat\psi_2\|_{L^2}\bigr)$.
Since $C^1_c(\mathcal A^{\circ})$ is dense in
$\{u\in L^2:\operatorname{supp}u\subset\mathcal A\}$ and
Proposition~\ref{prop:laprecovery} holds on the dense set, it holds throughout.
\end{proof}

For $A=-\Delta+V$ the density is no longer explicit, but the recovery persists
once the resolvent boundary values are supplied by the limiting absorption
principle. We isolate that principle as the imported ingredient; the dyadic
mechanism is unchanged.

\begin{proposition}[Schr\"odinger operators]\label{prop:schrod}
Let $n\ge3$ and $V:\R^n\to\R$ be measurable with
$|V(x)|\le C\langle x\rangle^{-1-\delta}$ for some $\delta>0$, where
$\langle x\rangle=(1+|x|^2)^{1/2}$, and let
$A=-\Delta+V$, self-adjoint on $H^2(\R^n)$. Then $\sigma_{\mathrm{ess}}(A)=[0,\infty)$,
$A$ has no eigenvalue in $(0,\infty)$, and the limiting absorption principle holds
\cite{Agmon,ReedSimonIV}: for $\sigma>\tfrac12$ and every compact
$I\Subset(0,\infty)$ the boundary values
$R_A(\lambda+i0)=\lim_{\varepsilon\to0^+}R_A(\lambda+i\varepsilon)$ exist in
$\mathcal B(L^2_\sigma,L^2_{-\sigma})$, the space of bounded linear operators from
$L^2_\sigma$ to $L^2_{-\sigma}$, locally uniformly on $I$; here
$L^2_{\pm\sigma}=L^2(\R^n,\langle x\rangle^{\pm2\sigma}\,dx)$. Let
$\psi\in L^2_\sigma$ with $\psi=E_A(I)\psi$ for some compact $I\Subset(0,\infty)$.
Then $\mu_{\psi,\psi}=\rho_\psi\,d\lambda$ on $(0,\infty)$, with
$\rho_\psi=\tfrac1\pi\operatorname{Im}\ip{R_A(\cdot+i0)\psi}{\psi}\in C_c(I)$, and
the recovery \eqref{eq:recovery} holds for every real-valued $f\in C^1_c((0,\infty))$, with $g$
replaced by $\rho_\psi$ in \eqref{eq:Tk}.
\end{proposition}

\begin{proof}
The shift identity \eqref{eq:shifted} holds for any self-adjoint operator, so
pairing it against $\psi$ produces \eqref{eq:Tk} with $d\mu_{\psi,\psi}=\rho_\psi\,d\lambda$
in place of $g\,dE$. By Kato's theorem the decay of $V$ excludes eigenvalues in
$(0,\infty)$ \cite{ReedSimonIV}, so there $\mu_{\psi,\psi}$ is purely absolutely
continuous; the limiting absorption principle then yields
$\rho_\psi=\tfrac1\pi\operatorname{Im}\ip{R_A(\cdot+i0)\psi}{\psi}\in C(I)$, and
$\psi=E_A(I)\psi$ forces $\operatorname{supp}\rho_\psi\subset I$ with
$\|\rho_\psi\|_{L^1}=\|\psi\|^2$. Step~1 of Proposition~\ref{prop:laprecovery} is
exactly this boundary convergence, tested against $f$; Steps~2--4 apply verbatim
with $\rho_\psi\in L^1$ of compact support in $(0,\infty)$ in place of $g$, since
they use only $\|\rho_\psi\|_{L^1}<\infty$, the compactness of
$\operatorname{supp}\rho_\psi$, and $f\in C^1_c$.
\end{proof}

\begin{remark}
Proposition~\ref{prop:schrod} is conditional on the cited limiting absorption
principle, which is the sole analytic input beyond the dyadic identity. That
principle already supplies the density
$\rho_\psi=\tfrac1\pi\operatorname{Im}\ip{R_A(\cdot+i0)\psi}{\psi}$, so the
content here is not that the dyadic data produce something otherwise
unavailable, but that the reconstruction is insensitive to how the density is
obtained: Steps~2--4 of Proposition~\ref{prop:laprecovery} use only
$\|\rho_\psi\|_{L^1}<\infty$, compact support in $(0,\infty)$, and $f\in C^1_c$,
none of which presupposes a closed form for $\rho_\psi$. The hypotheses are the standard short-range ones;
weakening them (slower decay, $n=1,2$ thresholds, or embedded eigenvalues) is
governed entirely by the corresponding refinements of the absorption principle,
not by the dyadic step.
\end{remark}

\subsection{Traces and the density of states}

The recovery above reconstructs the spectral measure $\mu_{\psi,\psi}$ of a single
vector. Summing over an orthonormal basis replaces it by
$\operatorname{Tr}E_A(\,\cdot\,)$, the spectral measure of the whole
operator---the density of states---and the dyadic identity descends to a trace
formula for it.

\begin{proposition}\label{prop:trace}
Let $A$ be self-adjoint with discrete spectrum $\{a_n\}$ (listed with
multiplicity) satisfying $\sum_n(1+|a_n|)^{-1}<\infty$. Then for $\lambda>0$ the
resolvent $R_A(i\lambda)$ is trace class and, with $p_n=\lambda+ia_n$,
\begin{equation}\label{eq:trace}
\operatorname{Tr}R_A(i\lambda)
=\sum_n\frac{1}{a_n-i\lambda}
=i\sum_n\Bigl[\frac{1}{1-e^{-p_n}}
-\sum_{k\ge1}\frac{2^{-k}}{1+e^{-p_n/2^{k}}}\Bigr],
\end{equation}
each bracket being the scalar identity \eqref{eq:recip}, equal to $1/p_n$. The
grouping is essential: the partial traces
$\operatorname{Tr}(1\mp e^{-\lambda T}U_T)^{-1}$ diverge separately.
\end{proposition}

\begin{proof}
Since $\sqrt{a_n^2+\lambda^2}\ge c_\lambda(1+|a_n|)$ with
$c_\lambda=\min(1,\lambda)/\sqrt2$, one has
$\sum_n|a_n-i\lambda|^{-1}\le c_\lambda^{-1}\sum_n(1+|a_n|)^{-1}<\infty$, so
$R_A(i\lambda)$ is trace class. Evaluating the trace in an eigenbasis
$\{\psi_n\}$ and using $U_t\psi_n=e^{-ita_n}\psi_n$ in
Proposition~\ref{prop:dyadic},
\[
\ip{R_A(i\lambda)\psi_n}{\psi_n}
=i\Bigl[\frac{1}{1-e^{-p_n}}-\sum_{k\ge1}\frac{2^{-k}}{1+e^{-p_n/2^{k}}}\Bigr]
=\frac{i}{p_n}=\frac{1}{a_n-i\lambda},
\]
the middle equality being \eqref{eq:recip}. Summing over $n$ gives
\eqref{eq:trace}. Finally $|(1-e^{-p_n})^{-1}|$ lies in
$[(1+e^{-\lambda})^{-1},(1-e^{-\lambda})^{-1}]$, bounded away from $0$, so
$\operatorname{Tr}(1-e^{-\lambda}U_1)^{-1}=\sum_n(1-e^{-p_n})^{-1}$ diverges;
likewise the remaining partial traces.
\end{proof}

\begin{remark}\label{rem:traceformula}
Reordering \eqref{eq:trace} by powers of the propagator gives the formal trace
formula
\[
\operatorname{Tr}R_A(i\lambda)
=i\sum_{j\ge1}e^{-j\lambda}\Theta(j)
-i\lim_{\ell\to\infty}\sum_{k=1}^{\ell}\sum_{j\ge1}(-1)^j2^{-k}e^{-j\lambda/2^{k}}\,
\Theta(j2^{-k}),\qquad \Theta(t)=\operatorname{Tr}e^{-itA},
\]
which expresses the spectral side through the propagator trace $\Theta$ sampled
on the dyadic time set $\{j2^{-k}\}$. The $j=0$ terms, each
$\operatorname{Tr}I=\infty$, cancel through $\sum_{k\ge1}2^{-k}=1$; as $\Theta$ is
distributional at real times, the formula is rigorous in the heat-regularized
reading $t\mapsto t-i0$. Equivalently, the traced Stone formula \eqref{eq:stoneim}
recovers the density of states $\sum_n\delta(\,\cdot\,-a_n)$ from the same dyadic
data---the global counterpart of Proposition~\ref{prop:laprecovery}.
\end{remark}

\begin{example}\label{ex:circle}
For $A=-d^2/dx^2$ on $L^2(\R/2\pi\Z)$ the eigenvalues are $a_n=n^2$, $n\in\Z$,
and $\sum_n(1+n^2)^{-1}<\infty$, so Proposition~\ref{prop:trace} applies. The left
side of \eqref{eq:trace} has the closed form
\[
\operatorname{Tr}R_A(i\lambda)=\sum_{n\in\Z}\frac{1}{n^2-i\lambda}
=\frac{\pi}{\sqrt{-i\lambda}}\,\coth\!\bigl(\pi\sqrt{-i\lambda}\bigr),
\]
while the propagator trace $\Theta(t)=\sum_{n\in\Z}e^{-itn^2}$ is a Jacobi theta
function---the same object produced by the Zak transform of a Gaussian in
Example~\ref{ex:gaussian}. Thus on the circle the trace relates a cotangent-type
sum to dyadic samples of a theta function; we verified \eqref{eq:trace} against
the closed form numerically.
\end{example}

\subsection{Spectral zeta functions}

The heat trace $H(t):=\operatorname{Tr}e^{-tA}=\sum_n e^{-ta_n}$ enters the spectral
zeta function $\zeta_A(\sigma)=\sum_n a_n^{-\sigma}$ through the Mellin transform
$\Gamma(\sigma)\zeta_A(\sigma)=\int_0^\infty t^{\sigma-1}H(t)\,dt$. The scalar
identity \eqref{eq:genpolylog} is an exact dyadic discretization of the Gamma
integral $\Gamma(\sigma)p^{-\sigma}=\int_0^\infty t^{\sigma-1}e^{-pt}\,dt$, so
summing it over the spectrum represents this Mellin transform by dyadic samples of
the heat trace. (Here $t$ is real: $H$ is the heat trace, in contrast with the
propagator trace $\Theta(t)=\operatorname{Tr}e^{-itA}$ of
Remark~\ref{rem:traceformula}.)

\begin{proposition}\label{prop:zeta}
Let $A$ be self-adjoint with discrete spectrum $a_n>0$, $a_n\to\infty$, and finite
heat trace $H(t)<\infty$ for all $t>0$, and let $\sigma_0$ be the abscissa of
convergence of $\zeta_A$. Then for $\Re\sigma>\sigma_0$,
\begin{equation}\label{eq:zeta}
\Gamma(\sigma)\,\zeta_A(\sigma)
=\sum_{j\ge1}j^{\sigma-1}H(j)
-\sum_{k\ge1}2^{-k\sigma}\sum_{j\ge1}(-1)^{j}\,j^{\sigma-1}\,H\!\bigl(j2^{-k}\bigr),
\end{equation}
and the first sum is an entire function of $\sigma$.
\end{proposition}

\begin{proof}
Take \eqref{eq:genpolylog} with $s=1-\sigma$, so that $\Gamma(1-s)=\Gamma(\sigma)$
and $p^{s-1}=p^{-\sigma}$, and evaluate at $p=a_n$:
\[
\Gamma(\sigma)\,a_n^{-\sigma}
=\Lis_{1-\sigma}\!\bigl(e^{-a_n}\bigr)
-\sum_{k\ge1}2^{-k\sigma}\Lis_{1-\sigma}\!\bigl(-e^{-a_n/2^{k}}\bigr).
\]
Sum over $n$. With $\Lis_{1-\sigma}(z)=\sum_{j\ge1}j^{\sigma-1}z^{j}$ and
$\sum_n e^{-ja_n}=H(j)$, the first term yields $\sum_{j\ge1}j^{\sigma-1}H(j)$;
since $e^{-ja_n}\le e^{-(j-1)a_1}e^{-a_n}$ gives $H(j)\le H(1)e^{-(j-1)a_1}$, this
series converges locally uniformly on $\C$ and is entire. For the second term,
$\Lis_{1-\sigma}(-x)=-x+O(x^2)$ near $0$ and is bounded on $[0,1]$, so
$|\Lis_{1-\sigma}(-x)|\le C_\sigma x$ there; hence
\[
\sum_{n}\sum_{k\ge1}2^{-k\Re\sigma}
\bigl|\Lis_{1-\sigma}\!\bigl(-e^{-a_n/2^{k}}\bigr)\bigr|
\le C_\sigma\sum_{k\ge1}2^{-k\Re\sigma}H\!\bigl(2^{-k}\bigr)<\infty
\qquad(\Re\sigma>\sigma_0),
\]
the last sum being comparable to the convergent Mellin integral
$\int_0^{1}t^{\sigma-1}H(t)\,dt$. This absolute convergence licenses interchanging
$\sum_n$ with $\sum_k$; expanding each alternating polylogarithm then produces the
inner sums $\sum_{j\ge1}(-1)^{j}j^{\sigma-1}H(j2^{-k})$, which proves
\eqref{eq:zeta}.
\end{proof}

\begin{remark}\label{rem:zetacont}
Identity \eqref{eq:zeta} splits $\zeta_A$ into an entire part and a meromorphic
remainder. The leading sum, built from the heat trace at integer times $t\ge1$
where $H$ decays exponentially, is entire; the dyadic correction probes $H$ at the
short times $j2^{-k}$ and carries the poles. If $H(t)\sim C\,t^{-\sigma_0}$ as
$t\to0^+$, the inner sums grow like $2^{k\sigma_0}$ and the outer geometric series
$\sum_k 2^{k(\sigma_0-\sigma)}=(1-2^{\sigma_0-\sigma})^{-1}$ converges exactly for
$\Re\sigma>\sigma_0$, its singularity at $\sigma=\sigma_0$ reproducing the leading
heat coefficient $\operatorname{Res}_{\sigma=\sigma_0}\zeta_A=C/\Gamma(\sigma_0)$.
Since $1/\Gamma(\sigma)$ vanishes at $\sigma=0$, the special values $\zeta_A(0)$
and $\zeta_A'(0)=-\log\det A$---the spectral determinant---are read from the
bracket near $\sigma=0$.
\end{remark}

\begin{example}\label{ex:zetacircle}
For $A=-d^2/dx^2$ on $L^2(\R/2\pi\Z)$ the positive eigenvalues are $n^2$,
$n\in\Z\setminus\{0\}$ (multiplicity two), so $\zeta_A(\sigma)=2\zeta(2\sigma)$
with $\zeta$ the Riemann zeta function, and
$H(t)=\sum_{n\neq0}e^{-tn^2}=\theta_3\!\bigl(0;e^{-t}\bigr)-1$, the same Jacobi
theta as in Example~\ref{ex:gaussian}. Here $\sigma_0=\tfrac12$, and the pole of
$\zeta_A$ at $\sigma=\tfrac12$ has residue $1$, matching
$C/\Gamma(\tfrac12)=\sqrt\pi/\sqrt\pi$ from $H(t)\sim\sqrt{\pi/t}$. Thus
\eqref{eq:zeta} is a dyadic-sampling form of the classical theta representation of
the Riemann zeta function; we verified \eqref{eq:zeta} itself, and the underlying
identity \eqref{eq:genpolylog}, numerically to $20$ and $30$ digits respectively.
\end{example}

\begin{example}[A functional determinant: $\det(-\Delta+m^{2})$ on the circle]
\label{ex:zetadet}
On $L^{2}(\R/L\Z)$ with $m>0$, the operator $A=-d^{2}/dx^{2}+m^{2}$ has eigenvalues
$a_n=(2\pi n/L)^{2}+m^{2}$, $n\in\Z$, all positive, and heat trace
$H(t)=e^{-tm^{2}}\theta_3\!\bigl(0;e^{-t(2\pi/L)^{2}}\bigr)$. Its short-time
expansion $H(t)=\tfrac{L}{2\sqrt{\pi}}\,t^{-1/2}e^{-tm^{2}}+(\text{exponentially
small})$ has no constant term, so $\zeta_A(0)=0$ and the zeta-regularized
determinant is
\[
\log\det A=-\zeta_A'(0)=-\bigl[\Gamma(\sigma)\zeta_A(\sigma)\bigr]_{\sigma=0},
\]
the bracket being regular at $\sigma=0$ by Remark~\ref{rem:zetacont}.

We read this value off \eqref{eq:zeta}. The entire part contributes
$E(0)=\sum_{j\ge1}H(j)/j$, an exponentially convergent sum of integer-time heat
samples. For the dyadic correction, subtract the leading term
$H_0(t)=\tfrac{L}{2\sqrt{\pi}}t^{-1/2}e^{-tm^{2}}$: the remainder $H-H_0$ is
exponentially small as $t\to0^{+}$ and yields a convergent
$D_{\mathrm{wind}}(0)=\sum_{k\ge1}\sum_{j\ge1}(-1)^{j}(H-H_0)(j2^{-k})/j$, while the
$H_0$ piece resums geometrically to
\[
D_0(0)=\frac{L}{2\sqrt{\pi}}\left[\frac{\Lis_{3/2}(-1)}{2^{-1/2}-1}
+\sum_{k\ge1}2^{k/2}\bigl(\Lis_{3/2}(-e^{-m^{2}2^{-k}})-\Lis_{3/2}(-1)\bigr)\right],
\]
whose first term is the geometric series $\sum_{k\ge1}2^{-k(\sigma-1/2)}$ evaluated
at $\sigma=0$---the very series whose pole at $\sigma=\tfrac12$ is isolated in
Remark~\ref{rem:zetacont}. Then $\log\det A=D_0(0)+D_{\mathrm{wind}}(0)-E(0)$.

This reproduces the classical closed form $\det A=4\sinh^{2}(mL/2)$. Across a range
of $(L,m)$ the reconstruction agrees to the precision of the truncated sums (about
$13$ digits); for instance $L=2\pi$, $m=1$ gives $\log\det A=6.2794469300261$,
matching $\log\!\bigl[4\sinh^{2}\pi\bigr]$. Figure~\ref{fig:zetadet} shows the
entire part converging exponentially in the number of integer heat samples---a
handful suffice, at a rate set by $m^{2}$---while remaining pole-free, the lone
pole at $\sigma_0=\tfrac12$ residing entirely in the dyadic correction.
\end{example}

\begin{figure}[ht]
\centering
\includegraphics[width=\textwidth]{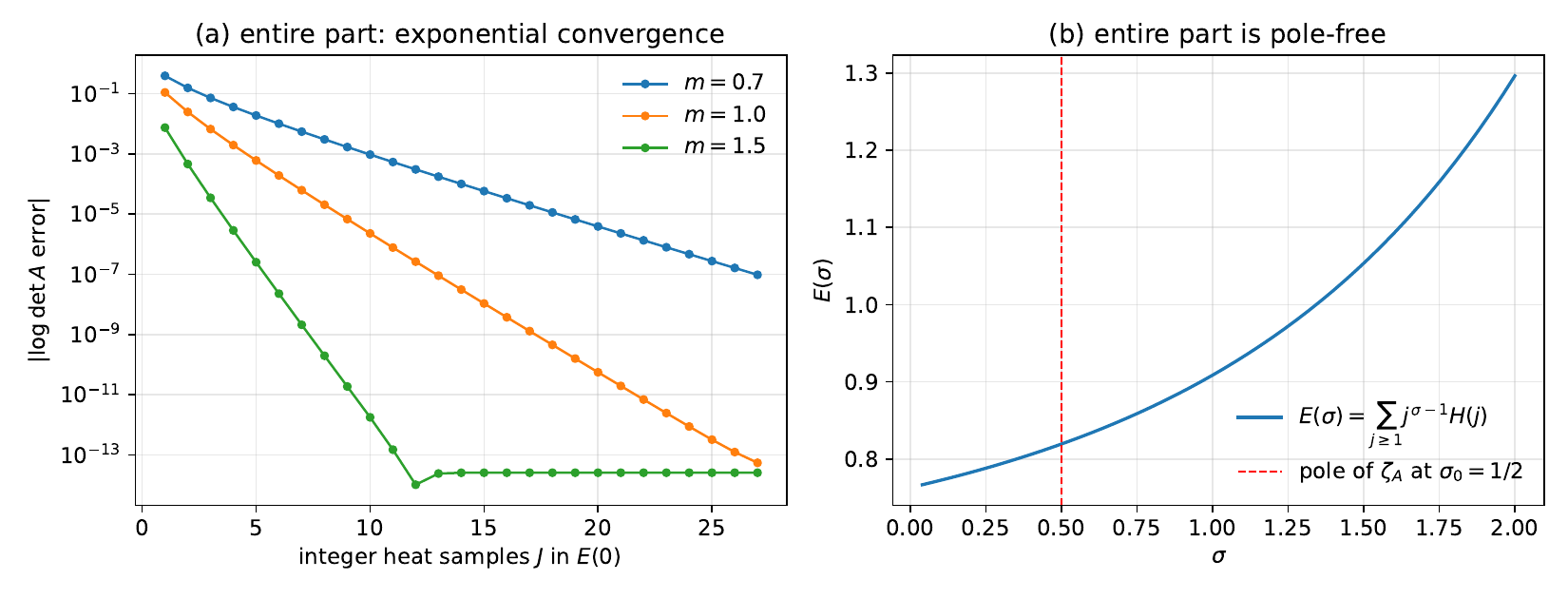}
\caption{Zeta-regularized determinant of $-\Delta+m^{2}$ on the circle via the
dyadic heat-trace split \eqref{eq:zeta}. (a)~Error of the reconstructed
$\log\det A$ against the closed form $4\sinh^{2}(mL/2)$, versus the number of
integer-time heat samples $H(1),\dots,H(J)$ retained in the entire part $E(0)$;
convergence is exponential with rate set by $m^{2}$, the floor near $10^{-14}$
being the working precision. (b)~The entire part
$E(\sigma)=\sum_{j\ge1}j^{\sigma-1}H(j)$ is smooth and pole-free; the only pole of
$\zeta_A$, at $\sigma_0=\tfrac12$, is carried by the dyadic correction.}
\label{fig:zetadet}
\end{figure}

\subsection{Identifiability from the dyadic samples}

The constructions of this paper run in a single direction: from the dyadic
propagator samples $\{U_{2^{-k}}\}_{k\ge0}$ to the resolvent through
\eqref{eq:dyadic}, and from the resolvent to the spectral measures, the density of
states, and the spectral zeta function through Theorem~\ref{thm:stone} and the
present section. Read backwards, they say that all of this spectral data is a
\emph{function of the dyadic samples alone}. We close by recording this as a
uniqueness statement, which promotes Proposition~\ref{prop:dyadic} from an identity
to an identifiability theorem.

Every summand on the right of \eqref{eq:dyadic} is a power
$U_{2^{-k}}^{\,j}=U_{j2^{-k}}$ of a single dyadic sample, so for each $\lambda>0$ the
operator $R_A(i\lambda)$ is determined by $\{U_{2^{-k}}\}_{k\ge0}$. The off-axis
resolvent is reached by modulation: since $e^{-it(A-aI)}=e^{iat}U_t$ for $a\in\R$,
the operator $A-aI$ has dyadic samples $e^{ia2^{-k}}U_{2^{-k}}$, and
\eqref{eq:dyadic} applied to it returns $R_A(a+i\lambda)=R_{A-aI}(i\lambda)$ from the
same data---the modulation already used for the energy dependence in
Proposition~\ref{prop:lattice}. Hence the dyadic samples determine $R_A(z)$ for every
$z$ with $\Im z\neq0$, and with it the spectral measures, density of states, and
zeta function of this section.

\begin{corollary}[Dyadic identifiability]\label{cor:identifiability}
Let $A$ and $B$ be self-adjoint operators on $\Hil$, with propagators
$U^A_t=e^{-itA}$ and $U^B_t=e^{-itB}$. If
\[
U^A_{2^{-k}}=U^B_{2^{-k}}\qquad\text{for every }k\ge0,
\]
then $A=B$. In particular $A$---together with all of its spectral measures, its
density of states, and its spectral zeta function---is determined by the single
sequence of dyadic propagator samples $\{U^A_{2^{-k}}\}_{k\ge0}$.
\end{corollary}

\begin{proof}
For every $j\ge0$ and $k\ge0$ one has
$U^A_{j2^{-k}}=(U^A_{2^{-k}})^{j}=(U^B_{2^{-k}})^{j}=U^B_{j2^{-k}}$, so the
right-hand side of \eqref{eq:dyadic} is identical for $A$ and $B$; therefore
$R_A(i\lambda)=R_B(i\lambda)$ in the strong operator topology for every $\lambda>0$.
Fix one such $\lambda$ and set $R=(A-i\lambda)^{-1}=(B-i\lambda)^{-1}$. This bounded
operator is a bijection of $\Hil$ onto $\operatorname{ran}R=D(A)=D(B)$, and on that
common domain $R^{-1}=A-i\lambda=B-i\lambda$; hence $A=B$, domains included. The
remaining assertions restate the constructions of \S\ref{sec:spectral}.
\end{proof}

\begin{remark}[Anti-aliasing]\label{rem:aliasing}
A single sample cannot suffice. The propagator $U_1=e^{-iA}$ is unchanged under
$A\mapsto A+2\pi I$, and more generally resolves the spectrum only modulo $2\pi$, so
no one $U_T$ separates spectral values differing by a multiple of $2\pi/T$. The
dyadic ladder removes exactly this ambiguity. On the spectrum,
$U^A_{2^{-k}}=U^B_{2^{-k}}$ for all $k$ forces $2^{-k}(a-b)\in2\pi\Z$ for every
$k\ge0$, that is $a-b\in\bigcap_{k\ge0}2\pi\,2^{k}\Z=\{0\}$. Concretely, a would-be
aliasing gap $a-b=2\pi m$ with $m=2^{p}q$, $q$ odd, passes undetected through every
coarser sample but is caught at scale $k=p+1$, where
$e^{-i2^{-k}(a-b)}=e^{-i\pi q}=-1$. The octave refinement of
Remark~\ref{rem:average} is thus an exact dealiasing of the propagator: each halving
of the clock fixes one more binary digit of the spectral gap, and the whole ladder
admits only $a=b$.
\end{remark}

\section*{Acknowledgments}

It is a pleasure to thank Ovidiu Costin and Rodica Costin. This work grew directly
out of our collaboration, and it reflects throughout their insight, their
generosity with ideas, and their patient mentorship of me both as a mathematician
and as a person. I am deeply grateful to them.

\end{document}